\documentclass[11pt]{amsart}
\usepackage[utf8]{inputenc}
\usepackage{amssymb,amsmath}
\usepackage{amsthm}
\usepackage{graphicx}
\usepackage{color}
\usepackage{tikz}
\usepackage{pgfplots}
\usepackage{enumitem}
\usepackage{algorithm}
\usepackage{setspace}
\usepackage{mathtools}
\usepackage{xfrac}

\usepackage{multicol} 
\usepackage{soul}
\usepackage{subcaption}
\usepackage{booktabs}
\usepackage{array}
\newcolumntype{L}[1]{>{\raggedright\let\newline\\\arraybackslash\hspace{0pt}}m{#1}}
\newcolumntype{C}[1]{>{\centering\let\newline\\\arraybackslash\hspace{0pt}}m{#1}}
\newcolumntype{R}[1]{>{\raggedleft\let\newline\\\arraybackslash\hspace{0pt}}m{#1}}
\usepackage{bbold}

\usepackage[hidelinks]{hyperref}
\usepackage{bbold}

\newcommand{\ddiv}{\operatorname{div}}
\newcommand{\tri}{\mathcal{T}}

\newcommand{\opid}{\mathtt{id}}

\newcommand{\opnorm}[1]{{\left\vert\kern-0.25ex\left\vert\kern-0.25ex\left\vert #1 
		\right\vert\kern-0.25ex\right\vert\kern-0.25ex\right\vert}}

\newcommand{\calI}{\mathcal{I}}

\newcommand{\frakB}{\mathfrak{B}}

\newcommand\dx{\,\text{d}x}

\newcommand{\dist}{\mathrm{dist}}

\newcommand{\AC}{\mathfrak{A}}

\newcommand{\Vh}{V_h}

\newcommand{\Ve}{\ensuremath{H^1_0(D)}} 

\newcommand{\IhK}{\calI_h^K}
\newcommand{\PipH}{\Pi_H^p}
\newcommand{\Nb}{\mathsf{N}}

\newcommand{\nK}{m_K} 

\newcommand{\LamKj}{\Lambda_{K,j}} 
\newcommand{\tLamKj}{\tilde\Lambda_{K,j}} 
\newcommand{\tLamKjl}{\tilde\Lambda^\ell_{K,j}}

\DeclareMathOperator*{\argmin}{arg\,min}

\newcommand{\vC}{\tilde{v}_H}

\newcommand{\uC}{\tilde{u}_H}

\newcommand{\vhC}{\tilde{v}_{H,h}}
\newcommand{\uhC}{\tilde{u}_{H,h}}

\newcommand{\uB}{\bar{u}_H}

\newcommand{\vH}{v_H}
\newcommand{\vK}{v_K}

\newcommand{\lamKl}{\lambda^\ell_K}

\newcommand{\lamvH}{\lambda_{\vH}}

\newcommand{\muH}{\mu_H}
\newcommand{\wH}{w_H}

\newcommand{\vh}{v_h}

\newcommand{\uh}{u_h}

\def\N{\mathbb{N}}

\def\R{\mathbb{R}}

\def\calR{\mathcal{R}}

\def\calT{\mathcal{T}}
\def\VHp{V_H^p}
\def\Vhp{V_{h,p^\prime}}
\def\tVHp{\tilde{V}_H^p}
\def\tVHpl{\tilde{V}_H^{p,\ell}}
\def\tVHhp{\tilde{V}_{H,h}^p}
\def\tVHhpl{\tilde{V}_{H,h}^{p,\ell}}

\definecolor{myBlue}{RGB}{30,144,255}
\definecolor{myGreen}{RGB}{69,169,0}
\definecolor{myRed}{RGB}{165,12,42} 
\definecolor{myOrange}{RGB}{225,92,22} 
\definecolor{color2}{RGB}{255, 126, 126}
\definecolor{color3}{RGB}{0, 100, 0}
\definecolor{color1}{RGB}{176, 226, 255}

\tikzset{cross/.style={cross out, draw=black, minimum size=2*(#1-\pgflinewidth), inner sep=0pt, outer sep=0pt},
	cross/.default={0.6ex}}

\newtheorem{theorem}{Theorem}[section]

\newtheorem{lemma}[theorem]{Lemma}
\newtheorem{corollary}[theorem]{Corollary}

\theoremstyle{definition}

\newtheorem{remark}[theorem]{Remark}

\allowdisplaybreaks
 
\pgfplotstableset{
	every head row/.style={before row=\toprule, after row=\midrule},
	every last row/.style={after row=\bottomrule},
}

\numberwithin{theorem}{section}
\numberwithin{equation}{section}
\numberwithin{table}{section}
\numberwithin{figure}{section}
\allowdisplaybreaks
\begin{document}
\title[A High-Order Multiscale Method]{A High-Order Approach to Elliptic Multiscale Problems with General Unstructured Coefficients}
\author[Roland~Maier]{Roland~Maier$^\dagger$}
\address{${}^{\dagger}$ Department of Mathematical Sciences, Chalmers University of Technology and University of Gothenburg, 412 96 Gothenburg, Sweden}
\email{rolandma@chalmers.se}
\date{\today}
%
%
\begin{abstract}
We propose a multiscale approach for an elliptic multiscale setting with general unstructured diffusion coefficients that is able to achieve high-order convergence rates with respect to the mesh parameter and the polynomial degree. The method allows for suitable localization and does not rely on additional regularity assumptions on the domain, the diffusion coefficient, or the exact (weak) solution as typically required for high-order approaches. Rigorous a priori error estimates are presented with respect to the involved discretization parameters, and the interplay between these parameters as well as the performance of the method are studied numerically.
\end{abstract}
%
%
\maketitle
\indent
{\tiny {\bf AMS subject classifications.}  
{\bf 65N12}, 
{\bf 65N30} 
} 

\section{Introduction}\label{s:intro}

Computational multiscale methods are popular tools to deal with microscopic features of partial differential equations (PDEs) that are typically encoded in an underlying material coefficient. It is well-known that standard finite element methods only achieve acceptable results if varying micro-features are resolved by the corresponding finite element mesh. Multiscale methods aim to overcome this problem and achieve good approximation properties already for coarse-level simulations at the cost of a moderate computational overhead. Prominent first-order approaches in the context of elliptic PDEs are the heterogeneous multiscale method~\cite{EE03,EE05,AbdEEV12}, (generalized) multiscale finite element methods~\cite{BabO83,BabCO94,HowW97,BabL11,EfeGH13}, adaptive local bases~\cite{GraGS12,Wey16}, and rough polyharmonic splines~\cite{OwhZB14}. 
Under appropriate smoothness assumptions on the material coefficient or the exact solution, higher-order multiscale methods have been considered for instance in~\cite{LiPT12,AbdB12} in connection with the heterogeneous multiscale method or in~\cite{HarPV13,AraHPV13}, known as multiscale hybrid-mixed methods.  
Other high-order approaches are the method presented in~\cite{AllB05,HesZZ14} related to the multiscale finite element method or the multiscale hybrid high-order method~\cite{CicEL19} that is designed to work in the context of general polytopal meshes. These methods, however, only yield high-order estimates if suitable smoothness assumptions hold. 
In the context of very general $L^\infty$-coefficients 
that do not allow one to exploit higher-order regularity of the exact solution, all the above-mentioned multiscale methods at most provide first-order convergence results in $H^1$. 

In this paper, we analyze a higher-order multiscale method for an elliptic model problem which is able to achieve high-order convergence rates for general unstructured $L^\infty$-coefficients  
only from additional (piecewise) smoothness assumptions on the force term. The construction is motivated by the localized orthogonal decomposition (LOD) method \cite{MalP14,HenP13}, and in particular its discontinuous version~\cite{ElfGMP13,ElfGM13}, as well as the multi-level construction in~\cite{Owh17} known as gamblets. 

The classical LOD method introduced in~\cite{MalP14} is a multiscale technique based on a~first-order conforming finite element space and a corresponding local quasi-interpolation operator that fulfills certain interpolation properties. This operator and its properties are key to deriving the first-order error estimates of the method. 
In general, one could generalize the idea and consider higher-order conforming discrete spaces as used in the context of higher-order finite element methods (see, e.g., \cite{BabG96,Sch98}). There even exist local quasi-interpolation operators for such spaces without restrictive smoothness assumptions \cite{Mel05} that fulfill properties similar to the ones required for the LOD method. For general non-smooth coefficients, however, high-order rates with respect to the mesh size can only be obtained if the quasi-interpolation operator fulfills additional orthogonality properties with respect to the $L^2$-scalar product. To the best of our knowledge, such a local quasi-interpolation operator is not known to date, and its construction might be a delicate task.
To overcome this difficulty, discontinuous discrete spaces are a suitable choice. The idea to use discontinuous functions traces back to \cite{ElfGMP13,ElfGM13}, where the authors proposed a first-order discontinuous Galerkin multiscale method in the spirit of the LOD. In~\cite{Owh17}, gamblets were introduced that provide an abstract setting that generalizes the LOD formulation. Their construction is based on constraint energy minimization problems with an arbitrarily chosen discrete space for the constraint conditions. In particular, one may employ discontinuous high-order finite element spaces. The potential of such spaces to enable higher-order convergence rates has not yet been exploited, but the approach has already been addressed in connection with higher-order differential operators in \cite{OwhS19}.

In the present work, we rely on a two-scale gamblet-construction in the spirit of~\cite{Owh17} combined with piecewise polynomials for the constraints. This results in a conforming discrete multiscale space despite the discontinuous constraint conditions. Our main focus lies on a thorough analysis of the corresponding Galerkin method in terms of the convergence behavior with respect to both the mesh size $H$ and the polynomial degree~$p$. We are able to prove that our (ideal) multiscale approximation~$u_\mathrm{ms}$ and the exact solution $u$ fulfill an error estimate of the form
\begin{equation*}
\|u - u_\mathrm{ms}\|_{H^1(D)} \leq C(s) \,\bigg(\frac{H}{p}\bigg)^{s+1} |f|_{H^s(\tri_H)}, \quad s \leq p+1, 
\end{equation*} 
with the sole requirement of a piecewise regular right-hand side $f$ with respect to the mesh~$\tri_H$; cf.~Theorem~\ref{t:errODp}. In particular, the error estimate holds under minimal regularity assumptions on the domain (Lipschitz), the diffusion coefficient ($L^\infty$), and the exact solution $u$ ($H^1$). We also show that a similar result can be retained for a fully discrete and localized variant of the method; cf.~Theorem~\ref{t:errfullydiscretep}.

The remaining parts of the paper are organized as follows. In Section~\ref{s:highorder}, we introduce the elliptic model problem and discontinuous coarse finite element spaces. Based on these spaces, we then construct a high-order multiscale space in Section~\ref{s:highorder} and analyze the corresponding ideal high-order multiscale method. In Section~\ref{s:practical}, we present a practical version of the method for which we finally provide numerical examples in Section~\ref{s:numexpp}. 

\subsection*{Notation} Throughout this work, we use the following notation. We write $C$ for any positive constant that is independent of the mesh sizes $H,\,h$ as well as the polynomial degree $p$, the localization parameter $\ell$, and the microscopic scale $\epsilon$. To indicate an explicit dependence on a parameter $\xi$, we may write $C_\xi$. We further abbreviate $a \leq C\, b$ by $a \lesssim b$ and use $a \sim b$ if $a \lesssim b$ and $a \gtrsim b$. 

\section{Problem Formulation and Discrete Spaces}

\subsection{Elliptic Model Problem}

We consider the prototypical second-order diffusion problem
\begin{equation}\label{eq:PDEell}
\begin{aligned}
-\ddiv (A \nabla u) &= f &&\text{ in } D,\\
u &= 0 &&\text{ on } \partial D,
\end{aligned}
\end{equation}
where $D \subseteq \R^d$, $d\in\{1,2,3\}$, is a bounded and polytopal Lipschitz domain and $f \in L^2(D)$. We assume the coefficient $A$ to encode microscopic features of the medium on some scale~$\epsilon$ and to be \emph{admissible}, i.e., it belongs to the set
\begin{equation}\label{eq:admissibleA}
\AC := \left\{\begin{aligned} 
&A \in L^\infty(D;\R_\mathrm{sym}^{d\times d})\,:\,\exists\, 0 < \alpha \leq \beta < \infty\,:\\&\forall \xi \in \R^d,\, \text{a.a. } x \in D\,:\, \alpha |\xi|^2 \leq A(x)\xi \cdot \xi \leq \beta |\xi|^2\end{aligned}\right\}
\end{equation}
with minimal assumptions. 
For a given coefficient $A \in \AC$, we write $\alpha$ for the largest possible choice of $\alpha$ in the definition \eqref{eq:admissibleA} and $\beta$ for the $L^\infty$-norm of $A$, i.e., $\beta = \|A\|_{L^\infty(D;\R_\mathrm{sym}^{d\times d})}$, although this choice of $\beta$ might not be the minimal constant with respect to the estimate in \eqref{eq:admissibleA}. We emphasize that also positive and bounded scalar coefficients are admissible, since these coefficients may simply be multiplied by the identity matrix. 

The variational formulation of \eqref{eq:PDEell} seeks a solution $u \in \Ve$ that solves
\begin{equation}\label{eq:PDEellweak}
a(u,v) = (f,v)_{L^2(D)}
\end{equation}
for all $v \in \Ve$, where 
\begin{equation*}
a(w,v) := \int_D A \nabla w \cdot \nabla v \dx, \quad w,\,v\in\Ve. 
\end{equation*}
Note that the solution $u$ of \eqref{eq:PDEellweak} is unique by the Lax-Milgram Theorem and it holds that
\begin{equation}\label{eq:stabuell}
\|\nabla u\|_{L^2(D)} \leq \alpha^{-1}\,\|f\|_{L^2(D)}.
\end{equation}
The aim of the multiscale construction in Section~\ref{s:highorder} below is to provide a suitable approximation of the solution $u$ in \eqref{eq:PDEellweak}. For the construction, we particularly require discontinuous high-order finite element spaces.

\subsection{Discontinuous discrete spaces}

Let $\{\tri_H\}_{H > 0}$ be a family of regular decompositions of the domain $D$ into quasi-uniform \emph{$d$-rectangles} on the scale $H$ as described in \cite[Ch.~2~\&~3]{Cia78}. Further, denote with $\VHp$ the space of piecewise polynomial functions with prescribed maximal coordinate degree $p$, i.e.,
\begin{equation*}
\VHp := \left\{\begin{aligned} 
v \in L^2(D) \,\colon\, \forall K \in \tri_H\,\colon\, v\vert_{{K}} \text{ is a polynomial}\\ 
\text{of coordinate degree $\leq p$}\end{aligned}\right\}.
\end{equation*}
For any $S \subseteq D$, we write $\VHp(S)$ for the restriction of $\VHp$ to the subdomain $S$. In particular, for any $K \in \tri_H$, the restricted space $\VHp(K)$ is exactly the space of polynomials up to degree $p$ in each coordinate direction on the element $K$.
For later use, we also define for $k \in \N$ the \emph{broken Sobolev space} $H^k(\tri_H)$ by
\begin{equation*}
H^k(\tri_H) := \{v \in L^2(D)\,\colon\, \forall K \in \tri_H\,\colon\, v\vert_K \in H^k(K)\}
\end{equation*}
with the seminorm 
\begin{equation*} 
|\cdot|_{H^k(\tri_H)} := \sum_{K \in \tri_H} |\cdot|^2_{H^k(K)}, 
\end{equation*}
where $|\cdot|_{H^k(S)} := \|\nabla^k \cdot\|_{L^2(S)}$ denotes the $H^k$-seminorm on $S\subseteq D$.

The next step towards our multiscale construction consists in defining a projection operator onto the space $\VHp$ that fulfills appropriate local stability and approximation properties. Here, we use the $L^2$-projection $\PipH\colon L^2(D)\to\VHp$ defined for any $v \in L^2(D)$ by
\begin{equation}\label{eq:defL2proj}
\big(\PipH v, \wH\big)_{L^2(D)} = \big(v, \wH\big)_{L^2(D)}
\end{equation}
for all $\wH \in \VHp$. The operator $\PipH$ is local due to the element-wise definition of the space~$\VHp$ and the possible discontinuities across element boundaries. That is, the definition of~$\PipH$ in \eqref{eq:defL2proj} is equivalent to the element-wise characterization
\begin{equation}\label{eq:defL2projloc}
\big({(\PipH v)\vert}_{K}, q\big)_{L^2(K)} = \big(v, q\big)_{L^2(K)}
\end{equation}
for all $q \in \VHp(K)$ and $K\in\tri_H$. 
For the sake of readability, we abbreviate $\Pi := \PipH$ if $p$ and $H$ are explicitly given and there is no possibility of confusion. 

For any $K \in \tri_H$, the $L^2$-stability of $\Pi$ follows directly from equation \eqref{eq:defL2projloc} with the choice $q = {(\Pi v)\vert}_{K}$ and reads
\begin{equation}\label{eq:PistabL2}
\|\Pi v\|_{L^2(K)} \leq \|v\|_{L^2(K)}
\end{equation}
for all $v\in L^2(K)$. 
Further, it holds that 
\begin{equation}\label{eq:Piapprox}
\|(\opid-\Pi) v\|_{L^2(K)} \leq C_{\Pi}\frac{H}{p}\, \|\nabla v\|_{L^2(K)}
\end{equation}
for all $v\in H^1(K)$; see, e.g., \cite{Sch98,HouSS02,Geo03}.
If $v \in H^k(K)$ for $k \in \N$ and $k \leq p+1$, we even have
\begin{equation}\label{eq:Piapproxk}
\|(\opid-\Pi) v\|_{L^2(K)} \leq C_{\Pi}\,\Phi(p,k)\,H^k\, |v|_{H^k(K)}
\end{equation}
with a constant $C_{\Pi}$ that does not depend on $H$ or $p$ and
\begin{equation*}
\Phi(p,k) := \bigg(\frac{(p+1-k)!}{(p+1+k)!}\bigg)^{1/2}.
\end{equation*}
We emphasize that due to the true locality of the inequalities \eqref{eq:PistabL2} and \eqref{eq:Piapprox}, the results immediately generalize to unions of elements and, in particular, to a~global result on the domain $D$ in the sense of an element-wise gradient on the right-hand side. 
At this point, we also introduce the inverse inequality for polynomials which states that 
\begin{equation}\label{eq:invineq}
\| \nabla q \|_{L^2(K)} \leq C_\mathrm{inv}H^{-1}p^2 \,\| q \|_{L^2(K)}
\end{equation}
for $K\in\tri_H$ and for all $q \in \VHp(K)$; 
see, e.g., \cite{Sch98,GraHS05,Geo08}. As above, this result also holds globally, i.e.,
\begin{equation*}
|\vH |_{H^1(\tri_H)} \leq C_\mathrm{inv}H^{-1}p^2 \,\| \vH \|_{L^2(D)}
\end{equation*}
for all $\vH \in \VHp$. 
We emphasize that $\Pi\colon L^2(D) \to \VHp$ is obviously surjective as an operator from $L^2(D)$ to the non-conforming space $\VHp$. For the following construction, however, we explicitly require that $\Pi$ is also surjective as an operator from $\Ve$ to $\VHp$. This very important condition is only mentioned here and rigorously proved in Section~\ref{ss:infsup} in order to improve the clarity of presentation.

\section{High-Order Multiscale Approximation}\label{s:highorder}

In this section, we state and analyze an ideal multiscale approach to discretize problem \eqref{eq:PDEellweak}. Therefore, we introduce a high-order multiscale space which is then used as discretization space for a Galerkin method.  

\subsection{Ideal trial and test space}

In the spirit of the LOD method and gamblets, we construct an operator $\calR\colon \VHp \to \Ve$ that assigns to each $\vH\in \VHp$ a function in $\Ve$ whose $L^2$-projection is exactly $\vH$. Such functions exist by the surjectivity of $\Pi\vert_{\scalebox{.65}{$\Ve$}}$ (see Theorem~\ref{t:surjectivePi} and Corollary~\ref{c:locbubble} below) but we particularly want the space $\calR\VHp$ to have improved approximation properties compared to a classical finite element space for which error estimates typically depend on the scale of microscopic oscillations. 

We define $\calR\colon \VHp \to \Ve$ for any $\vH \in\VHp$ as the solution of the saddle point problem
\begin{equation}\label{eq:saddlepoint}
\begin{aligned}
a(\calR \vH, v)\qquad\, &+& (\lamvH, v)_{L^2(D)} \quad&=\quad 0,\\
(\calR \vH,\muH)_{L^2(D)} && &=\quad (\vH, \muH)_{L^2(D)}
\end{aligned}
\end{equation}
for all $v \in \Ve$ and all $\muH \in \VHp$, where $\lamvH \in \VHp$ is the associated Lagrange multiplier. The solution $(\calR \vH,\lamvH) \in \Ve \times \VHp$ of \eqref{eq:saddlepoint} exists and is unique due to the equivalent definition
\begin{equation*}
\calR \vH := \argmin_{v \in \Ve}\, a(v,v) \quad\text{ subject to }\quad\Pi v =  \vH,
\end{equation*}
which is well-posed by the surjectivity of $\Pi\vert_{\scalebox{.65}{$\Ve$}}$.
We now set $\tVHp := \calR\VHp \subseteq \Ve$ and observe that $\dim \tVHp = \dim \VHp$ because ${\calR\colon \VHp \to \tVHp}$ is a bijection with inverse $\Pi\vert_{\scalebox{.65}{$\tVHp$}}$. 

\subsection{The ideal method}\label{ss:idealmethod}

Using $\tVHp$ as test and trial space for a continuous Galerkin approach, we obtain a multiscale method which computes a finite-dimensional approximation of \eqref{eq:PDEellweak}. This so-called \emph{ideal method} reads: find $\uC \in \tVHp$ such that 
\begin{equation}\label{eq:PDEellOD}
a(\uC,\vC) = (f,\vC)_{L^2(D)}
\end{equation}
for all $\vC \in \tVHp$. As for the variational problem \eqref{eq:PDEellweak}, we directly get the well-posedness of \eqref{eq:PDEellOD} from the Lax-Milgram Theorem due to the conformity of $\tVHp$. 

In the following theorem, we quantify the error between the solutions of \eqref{eq:PDEellweak} and~\eqref{eq:PDEellOD} under additional (piecewise) regularity assumptions on the right-hand side $f$ and independently of possible oscillations of the coefficient $A$. 

\begin{theorem}[Error of the ideal method]\label{t:errODp}
Assume that $f \in H^k(\tri_H),\, k \in \N_0$, and define $s := \min\{k,p+1\}$. Further, let $u \in \Ve$ and $\uC \in \tVHp$ be the solutions of \eqref{eq:PDEellweak} and \eqref{eq:PDEellOD}, respectively. Then
\begin{equation}\label{eq:errODp}
\|\nabla(u - \uC)\|_{L^2(D)} \lesssim \frac{\Phi(p,s)}{p}\,H^{s+1}\, |f|_{H^s(\tri_H)}
\end{equation}
and
\begin{equation}\label{eq:errODpL2}
\|u - \uC\|_{L^2(D)} \lesssim \frac{\Phi(p,s)}{p^2}\,H^{s+2}\, |f|_{H^s(\tri_H)},
\end{equation}
with the notation $H^0(\tri_H) := L^2(D)$ and $|\cdot|_{H^0(\tri_H)}:= \|\cdot\|_{L^2(D)}$.
\end{theorem}

In order to prove Theorem~\ref{t:errODp}, we need the following useful result. 

\begin{lemma}[Equal projections]\label{l:equivform}
Let $u \in \Ve$ be the solution of \eqref{eq:PDEellweak} and $\uC \in \tVHp$ the solution of \eqref{eq:PDEellOD}. Then, we have that
\begin{equation*}
\Pi \uC = \Pi u.
\end{equation*}
\end{lemma}

\begin{proof}
We show that $\uC =\calR \Pi u$ or, equivalently, that $\uC$ solves
\begin{equation}\label{eq:saddlepointu}
\begin{aligned}
a(\uC, v)\qquad\, &+& (\lambda_{\Pi u}, v)_{L^2(D)} \quad&=\quad 0,\\
(\Pi\tilde{u}_H,\muH)_{L^2(D)} && &=\quad (\Pi u, \muH)_{L^2(D)}
\end{aligned}
\end{equation}
for all $v\in \Ve$ and $\muH \in\VHp$, where $\lambda_{\Pi u} \in \VHp$ is the associated Lagrange multiplier. The assertion then follows by the second line of \eqref{eq:saddlepointu}.

For $\vC = \calR \vH \in \tVHp$, we compute 
\begin{equation*}
\begin{aligned}
a(\calR\Pi u, \vC) &= a(u,\vC) - a((\opid-\calR\Pi)u,\vC)\\
& = (f,\vC)_{L^2(D)} - a((\opid-\calR\Pi)u,\vC).
\end{aligned}
\end{equation*}
Since $\Pi (\opid-\calR\Pi)u  = 0$, we get with \eqref{eq:saddlepoint} that 
\begin{equation*}
a((\opid-\calR\Pi)u,\vC) = 0.
\end{equation*}
Therefore, $\calR\Pi u$ is the unique solution of problem \eqref{eq:PDEellOD}. 
\end{proof}

\begin{proof}[Proof of Theorem~\ref{t:errODp}]
Using the Galerkin orthogonality, \eqref{eq:Piapprox}, and \eqref{eq:Piapproxk}, we obtain for $k \geq 1$ 
\begin{equation*}
\begin{aligned}
\alpha\,\|\nabla(u - \uC)\|^2_{L^2(D)} &\leq a(u - \uC,u - \uC) = a(u,u - \uC)\\
& = (f, u - \uC)_{L^2(D)} =
(f - \Pi f, u - \uC)_{L^2(D)} \\
& \leq \|f - \Pi f\|_{L^2(D)}\,C_{\Pi}\,\frac{H}{p}\,\|\nabla(u - \uC)\|_{L^2(D)}\\
& \leq C_{\Pi}\,\Phi(p,s)\,H^s\, |f|_{H^s(\tri_H)}\,C_{\Pi}\,\frac{H}{p}\,\|\nabla(u - \uC)\|_{L^2(D)},
\end{aligned}
\end{equation*}
where we employ that $\Pi (u - \uC) = 0$ by Lemma~\ref{l:equivform}. Thus, 
\begin{equation*}
\|\nabla(u - \uC)\|_{L^2(D)} \leq \alpha^{-1}C_{\Pi}^2\,\frac{\Phi(p,s)}{p}\,H^{s+1}\, |f|_{H^s(\tri_H)}.
\end{equation*}
With the same arguments but without inserting $\Pi f$, we get in the case $k = 0$ that
\begin{equation*}
\|\nabla(u - \uC)\|_{L^2(D)} \leq \alpha^{-1}C_{\Pi}\,\frac{H}{p}\, \|f\|_{L^2(D)}.
\end{equation*}
This proves \eqref{eq:errODp}. 
To show the $L^2$-error estimate, we use once again that ${\Pi (u - \uC) = 0}$. 
Therefore, we get with \eqref{eq:Piapprox} that
\begin{equation*}
\|u - \uC\|_{L^2(D)} \leq C_{\Pi}\,\frac{H}{p}\,\|\nabla(u-\uC)\|_{L^2(D)}.
\end{equation*}
Combining the last estimate with \eqref{eq:errODp}, we deduce \eqref{eq:errODpL2}. 
\end{proof}

\begin{remark}
If $k=0$ and $p = 1$, the error estimate in Theorem~\ref{t:errODp} is comparable to the error estimates of the conforming (ideal) LOD method; see, e.g.,~\cite{MalP14}.
\end{remark}

\subsection{Surjectivity of $\Pi$}\label{ss:infsup}

This subsection is devoted to showing that the projection operator~$\Pi$ is surjective when restricted to functions in $\Ve$. This property is an important requirement in the above construction and is proved for completeness. First, we show the following auxiliary lemma. 
\begin{lemma}[Local inf-sup condition]\label{l:infsup}
Let $K \in \tri_H$. Then the inf-sup condition
\begin{equation}\label{eq:infsupL2}
\adjustlimits\inf_{q \in \VHp(K)}\sup_{v \in H^1_0(K)} \frac{(q,v)_{L^2(K)}}{\|q\|_{L^2(K)}\,\|\nabla v\|_{L^2(K)}} \geq \gamma(H,p) > 0
\end{equation}
holds with $\gamma(H,p) \sim Hp^{-2}$. 
\end{lemma}

\begin{proof}
Let $\kappa \subseteq K$ be such that the edges and faces of $\kappa$ are parallel to the ones of~$K$. According to \cite[Lem.~3.7]{Geo08}, there exists a choice of $\kappa$ such that $\dist(\kappa,\partial K) = C_{\dist}\,Hp^{-2}$ and
\begin{equation}\label{eq:pol}
\| q \|^2_{L^2(\kappa)} \geq \frac{1}{4}\| q \|^2_{L^2(K)}
\end{equation}
for all $q \in \VHp(K)$, where $\dist(\cdot,\cdot)$ denotes the Hausdorff distance. 
As a next step, let ${\rho \in W^{1,\infty}(K)\cap H^1_0(K)}$ be a bubble function with 
\begin{equation*}
\begin{aligned}
&0 \,\leq \,\rho\,\leq\,1,\\
&\rho\,\equiv\,1\quad\text{in }\kappa,\\
&\|\nabla\rho\|_{L^\infty(K)} \,\leq \,C_{\rho}\,H^{-1}p^2,
\end{aligned}
\end{equation*}
where $C_{\rho}$ depends on $C_{\dist}$. Using \eqref{eq:pol} and 
\begin{equation}\label{eq:pol2}
\begin{aligned}
\|\nabla(\rho q)\|_{L^2(K)} &\leq \|\nabla\rho\|_{L^\infty(K)}\,\|q\|_{L^2(K)} + \|\rho\|_{L^\infty(K)}\,\|\nabla q\|_{L^2(K)}\\
&\leq H^{-1}p^2(C_\rho + C_\mathrm{inv})\,\|q\|_{L^2(K)}, 
\end{aligned}
\end{equation}
we get for any $q \in \VHp(K)$ that  
\begin{align*}
\sup_{v \in H^1_0(K)} \frac{(q,v)_{L^2(K)}}{\|q\|_{L^2(K)}\,\|\nabla v\|_{L^2(K)}} &\geq 
\frac{(q,\rho q)_{L^2(K)}}{\|q\|_{L^2(K)}\,\|\nabla(\rho q)\|_{L^2(K)}}\\
&\geq \frac{1}{4}\frac{\|q\|^2_{L^2(K)}}{\|q\|_{L^2(K)}\,\|\nabla(\rho q)\|_{L^2(K)}}\\
&= \frac{H}{4p^2\,(C_\rho + C_\mathrm{inv})} =: \gamma(H,p) > 0. \qedhere
\end{align*}
\end{proof}

\begin{theorem}[Surjectivity]\label{t:surjectivePi}
The restricted operator ${\Pi\vert}_{\scalebox{.65}{$\Ve$}}$ is surjective, i.e., for any $\wH \in \VHp$, there exists a function $w \in \Ve$ such that $\Pi w = \wH$. Further, among all possible candidates exists a choice of $w$ such that
\begin{equation*}
\|\nabla w\|_{L^2(D)} \lesssim \frac{p^2}{H}\,\|\wH\|_{L^2(D)}.
\end{equation*}
\end{theorem}

\begin{proof}
Let $\wH \in \VHp$. We define $w\in\Ve$ as the solution of
\begin{equation}\label{eq:defbubble}
\begin{aligned}
a(w, v)\qquad\, &+& (\lambda_{\wH}, v)_{L^2(D)} \quad&=\quad 0,\\
(w,\muH)_{L^2(D)} && &=\quad (\wH, \muH)_{L^2(D)}
\end{aligned}
\end{equation}
for all $v \in \Ve$ and all $\muH \in \VHp$.
From classical saddle point theory (see, e.g., \cite[Cor.~4.2.1]{BofBF13}), we know that \eqref{eq:defbubble} has a unique solution if the inf-sup condition
\begin{equation}\label{eq:infsupL2global}
\adjustlimits\inf_{\vH \in \VHp}\sup_{v \in H^1_0(D)} \frac{(\vH,v)_{L^2(D)}}{\|\vH\|_{L^2(D)}\,\|\nabla v\|_{L^2(D)}} \geq \tilde\gamma(H,p) > 0
\end{equation}
holds and $a$ is coercive. To show the inf-sup condition \eqref{eq:infsupL2global}, let $\vH \in \VHp$. From the construction in the proof of Lemma~\ref{l:infsup}, we get for any $K \in \tri_H$ the existence of a function $\vK \in H^1_0(K)$ which fulfills 
\begin{equation}\label{eq:est1}
(\vH,\vK)_{L^2(K)} \gtrsim \|\vH\|^2_{L^2(K)}
\end{equation}
and similarly to \eqref{eq:pol2} also
\begin{equation}\label{eq:est2}
\|\nabla\vK\|_{L^2(K)} \lesssim H^{-1}p^2\,\|\vH\|_{L^2(K)}. 
\end{equation}
Using these local contributions, the inclusion
\begin{equation*}
\bigcup_{K \in \tri_H} H^1_0(K) \subseteq H^1_0(D),
\end{equation*}
and the estimates \eqref{eq:est1} and \eqref{eq:est2}, we compute
\begin{equation*}
\begin{aligned}
\sup_{v \in H^1_0(D)} \frac{(\vH,v)_{L^2(D)}}{\|\vH\|_{L^2(D)}\,\|\nabla v\|_{L^2(D)}} &\geq \frac{\sum_{K \in \tri_H}(\vH,\vK)_{L^2(K)}}{\|\vH\|_{L^2(D)}\big(\sum_{K \in \tri_H}\|\nabla \vK\|^2_{L^2(K)}\big)^{1/2}}\\
&\geq C\,Hp^{-2}\frac{\sum_{K \in \tri_H}\|\vH\|^2_{L^2(K)}}{\|\vH\|_{L^2(D)}\big(\sum_{K \in \tri_H}\|\vH\|^2_{L^2(K)}\big)^{1/2}}
= C\,Hp^{-2}.
\end{aligned}
\end{equation*}
That is, the inf-sup condition \eqref{eq:infsupL2global} holds with ${\tilde\gamma(H,p) \sim Hp^{-2}} > 0$. Thus, \eqref{eq:defbubble} is well-posed and the stability estimates
\begin{equation*}
\|\lambda_{\wH}\|_{L^2(D)} \leq \frac{\beta}{\tilde\gamma(H,p)^2}\,\|\wH\|_{L^2(D)}
\end{equation*}
and
\begin{equation*}
\|\nabla w\|_{L^2(D)} \leq \frac{2\beta^{1/2}}{\alpha^{1/2}\tilde\gamma(H,p)}\,\|\wH\|_{L^2(D)}
\end{equation*}
hold (cf. \cite[Cor.~4.2.1]{BofBF13}). 
Finally, we remark that the equality $\Pi w = \wH$ follows by construction.
\end{proof}

The construction in the proof of Theorem~\ref{t:surjectivePi} is based on local subspaces of $\Ve$ and, thus, allows us to even find a conforming preimage $w\in\Ve$ under $\Pi$ of a function \mbox{$\wH \in \VHp$} which is supported only in the elements where $\wH$ is non-zero. This straightforward consequence is given in the following corollary.

\begin{corollary}[Local bubble function]\label{c:locbubble}
Let $\{K_j\}_{j = 1}^{n_R} \subseteq \calT_H$ be a set of elements and $\wH \in \VHp$ such that
\begin{equation*}
{\wH\vert}_{D\setminus R} = 0, \quad\text{ where }\quad R = \bigcup_{j=1}^{n_R} K_j.
\end{equation*} 
Then there exists a function $w \in H^1_0(R)$ with ${w\vert}_{D\setminus R} = 0$ such that $\Pi w = \wH$ and 
\begin{equation*}
\|\nabla w\|_{L^2(R)} \lesssim \frac{p^2}{H}\, \|\wH\|_{L^2(R)}.
\end{equation*}
\end{corollary}

\section{Derivation of a Practical Method}\label{s:practical}

In this section, we derive a fully practical version of the ideal method given in~\eqref{eq:PDEellOD} following the ideas of the classical LOD. The requirement of a practical version of the above method is related to the fact that the construction of the finite-dimensional approximation space $\tVHp$ involves the solution of infinite-dimensional problems. Therefore, we first investigate the decay properties of functions in $\tVHp$ which then allows us to suitable localize and discretize the construction of the space $\tVHp$.

\subsection{Decay of the basis functions}

As a first step, we identify a suitable choice of a basis of $\tVHp$ which is constructed from a basis of $\VHp$. For any $K \in \tri_H$, let 
\begin{equation*}
\frakB_K :=\{ \Lambda_{K,j} \}_{j = 1}^{\nK} \quad\text{ with }\quad \nK = (p+1)^d
\end{equation*}
be a basis of $\VHp(K)$ and
\begin{equation*}
\frakB := \bigcup_{K \in \tri_H} \frakB_K
\end{equation*}
the corresponding (local) basis of $\VHp$.  
In our numerical computations, we choose shifted Legendre polynomials on each element $K$, which are orthogonal with respect to the $L^2$-scalar product $(\cdot,\cdot)_{L^2(K)}$. 

Using the isomorphism $\calR$ between $\VHp$ and $\tVHp$, we directly get that $\tilde\frakB:=\calR\frakB$ is a basis of $\tVHp$. In the following, we show that for any basis function $\Lambda \in \frakB$, the corresponding basis function $\calR\Lambda \in \tilde\frakB$ decays exponentially fast away from the support of the function~$\Lambda$, which is exactly one element of $\tri_H$. 
To this end, we define for $\ell \in \N$ the \emph{element patch of order $\ell$} around $S \subseteq D$ by
\begin{equation*}
\Nb^\ell(S) := \Nb^1(\Nb^{\ell-1}(S)),\quad \ell \geq 1,\qquad\Nb^1(S) := \bigcup \bigl\{K \in \tri_H\,\colon\, \overline{S} \,\cap\, \overline{K}\neq \emptyset\bigl\}.
\end{equation*} 

\begin{theorem}[Decay of the basis functions]\label{t:decay}
Let $\ell \in \N$, $K \in \calT_H$, and $\Lambda \in \frakB_K$. Further, define $\tilde{\Lambda} = \calR\Lambda \in \tilde\frakB$. Then it holds that 
\begin{equation}\label{eq:decayLam}
\|\nabla \tilde\Lambda\|_{L^2(D\setminus \Nb^\ell(K))} \lesssim \exp(-C_\mathrm{dec}\, \ell/p) \,\|\nabla\tilde\Lambda\|_{L^2(D)}
\end{equation}
with a constant $C_\mathrm{dec}$ that depends on $C_\Pi$, $\alpha$, and $\beta$.
\end{theorem}

\begin{proof}
We choose a cutoff function $\eta \in W^{1,\infty}(D)$ with the following properties,
\begin{equation}\label{eq:defeta}
\begin{aligned}
&0 \leq \eta \leq 1,\\
&\eta \,= \,0\quad\text{in }\Nb^{\ell}(K),\\
&\eta\,=\,1\quad\text{in }D\setminus\Nb^{\ell+1}(K),\\
&\|\nabla\eta\|_{L^\infty(D)} \,\leq\, C_\eta \,H^{-1}.
\end{aligned}
\end{equation}
Define $R := \Nb^{\ell+1}(K)\setminus \Nb^{\ell}(K)$. Since $R$ is a union of elements of $\tri_H$ and ${\Pi (\tilde\Lambda\eta)\vert}_{D\setminus R} = 0$, we know from Corollary~\ref{c:locbubble} that there exists a bubble function $b\in H^1_0(R)$ which fulfills $\Pi b = \Pi (\tilde\Lambda\eta)$ and
\begin{equation}\label{eq:estbubble}
\|\nabla b\|_{L^2(R)} \lesssim \frac{p^2}{H} \|\tilde\Lambda\eta\|_{L^2(R)}.
\end{equation} 
We compute
\begin{equation*}
\begin{aligned}
\alpha\,\|\nabla\tilde\Lambda\|^2_{L^2(D\setminus\Nb^{\ell+1}(K))} &
\leq \Big|\int_D A \nabla\tilde\Lambda\cdot\nabla(\tilde\Lambda\eta)\dx\Big| + \Big|\int_D A \nabla\tilde\Lambda \cdot \nabla\eta\, \tilde\Lambda\dx\Big|\\
&\leq \Big|\int_D A \nabla\tilde\Lambda\cdot\nabla(\tilde\Lambda\eta - b)\dx\Big| \\
&\qquad+ \Big|\int_D A \nabla\tilde\Lambda\cdot\nabla b \dx\Big| + \Big|\int_D A \nabla\tilde\Lambda \cdot \nabla\eta\, \tilde\Lambda\dx\Big|\\
& = \Big|\int_R A \nabla\tilde\Lambda\cdot\nabla b\dx\Big| + \Big|\int_R A \nabla\tilde\Lambda \cdot \nabla\eta\, \tilde\Lambda\dx\Big|,
\end{aligned}
\end{equation*}
where we use the fact that by \eqref{eq:saddlepoint}, we have $a(\tilde\Lambda,v) = 0$ for any $v\in\Ve$ with $\Pi v = 0$. 
Therefore, we get with \eqref{eq:Piapprox}, ${\Pi \tilde\Lambda\vert}_{R}=0$, \eqref{eq:defeta}, and \eqref{eq:estbubble} that
\begin{equation*}
\|\nabla\tilde\Lambda\|^2_{L^2(D\setminus\Nb^{\ell+1}(K))} 
\leq Cp\,\|\nabla\tilde\Lambda\|^2_{L^2(R)}.
\end{equation*}
Employing the identity 
\begin{equation*}
R = \Nb^{\ell+1}(K)\setminus\Nb^{\ell}(K) = \big(D\setminus\Nb^{\ell}(K)\big)\setminus \big(D\setminus\Nb^{\ell+1}(K)\big),
\end{equation*}
we obtain
\begin{equation*}
\|\nabla\tilde\Lambda\|^2_{L^2(D\setminus\Nb^{\ell+1}(K))} \leq Cp\, \|\nabla\tilde\Lambda\|^2_{L^2(D\setminus\Nb^{\ell}(K))} - Cp\, \|\nabla\tilde\Lambda\|^2_{L^2(D\setminus\Nb^{\ell+1}(K))}
\end{equation*}
and thus 
\begin{equation*}
\|\nabla\tilde\Lambda\|^2_{L^2(D\setminus\Nb^{\ell+1}(K))} \leq \frac{Cp}{Cp+1}\,
\|\nabla\tilde\Lambda\|^2_{L^2(D\setminus\Nb^\ell(K))} \leq  \left(\frac{Cp}{Cp+1}\right)^{\ell+1} \|\nabla\tilde\Lambda\|^2_{L^2(D)}
\end{equation*}
with an iteration of the above arguments. We further get
\begin{equation*}
\Big(\frac{Cp}{Cp+1}\Big)^{\ell} = \exp\big(-|\log\big(\tfrac{Cp}{Cp+1}\big)|\,\ell\big) \leq 
\exp\big(-\tfrac{1}{2C}\,\ell/p\big). 
\end{equation*}
Taking the square root, we deduce \eqref{eq:decayLam} with $C_\mathrm{dec} := \frac{1}{4C}$.
\end{proof}

\begin{remark}\label{r:decay}
Although Theorem~\ref{t:decay} only quantifies the decay of basis functions $\tilde\Lambda \in \tilde\frakB$, with the same arguments the result also holds for any function $\calR q$, where $q \in \VHp(K)$ and $K \in \tri_H$. That is, we have  
\begin{equation}\label{eq:decay}
\|\nabla \calR q\|_{L^2(D\setminus \Nb^\ell(K))} \lesssim \exp(-C_\mathrm{dec}\, \ell/p) \,\|\nabla \calR q\|_{L^2(D)}.
\end{equation}
\end{remark}

\begin{remark}\label{r:pdep}
The $p$-dependence in Theorem~\ref{t:decay} seems pessimistic and is possibly not sharp, which is related to the mismatch between the interpolation estimate~\eqref{eq:Piapprox} and the inverse inequality~\eqref{eq:invineq} in terms of powers of $p$. In fact, the numerical experiments in Section~\ref{s:numexpp} indicate a scaling in $p$ which leads to a faster decay in \eqref{eq:decayLam} and \eqref{eq:decay} for increasing values of $p$. More precisely, one may expect that the exponential factor in these estimates is actually given by $\exp(-C_\mathrm{dec}\,\ell \,p^\delta)$ for some $\delta > 0$. That is, one could fix the parameter $\ell$ and reduce the localization error only by an adjustment of the polynomial degree.
\end{remark}

The decay property of the basis functions in $\frakB$ that is proved in Theorem~\ref{t:decay} is the key ingredient to define a localized version of the operator $\calR$. This localization procedure is explained and investigated in the following subsection.

\subsection{Localized computation of the approximation space}

We base the definition of a localized operator $\calR^\ell$ on truncated versions of the basis functions in $\tilde\frakB$. Thus, for a given localization parameter $\ell \in \N$ and any $\Lambda \in \frakB$ with ${\mathrm{supp}(\Lambda)= K\in\tri_H}$, we define ${\tilde\Lambda^\ell \in H^1_0(\Nb^\ell(K))}$ as the unique solution of the saddle point problem
\begin{equation}\label{eq:basisfunctionloc}
\begin{aligned}
a(\tilde\Lambda^\ell, v)\qquad\, &+& (\lambda_\Lambda^\ell, v)_{L^2(D)} \quad&=\quad 0,\\
(\tilde\Lambda^\ell,\muH)_{L^2(D)} && &=\quad (\Lambda, \muH)_{L^2(D)}
\end{aligned}
\end{equation}
for all $v \in H^1_0(\Nb^\ell(K))$ and $\muH \in \VHp(\Nb^\ell(K))$ with the associated Lagrange multiplier \mbox{$\lambda_\Lambda^\ell\in \VHp(\Nb^\ell(K))$}. Then for any function $\vH \in \VHp$ which can be expanded as
\begin{equation*}
\vH = \sum_{K \in\tri_H}\sum_{j=1}^{\nK} c_{K,j}\, \LamKj,
\end{equation*}
we define the corresponding function $\calR^\ell\vH \in H^1_0(D)$ by
\begin{equation}\label{eq:Rlocal}
\calR^\ell\vH := \sum_{K \in\tri_H}\sum_{j=1}^{\nK} c_{K,j}\, \tLamKjl.
\end{equation}
We set $\tVHpl := \calR^\ell\VHp$ and remark that $\tilde\frakB^\ell:=\calR^\ell\frakB$ is a basis of $\tVHpl$ by construction. We use this space to compute an approximation of the ideal finite-dimensional solution $\uC \in \tVHp$ of \eqref{eq:PDEellOD}, i.e., we want to find $\uC^\ell \in \tVHpl$ that solves
\begin{equation}\label{eq:PDEellLOD}
a(\uC^\ell,\vC) = (f,\vC)_{L^2(D)}
\end{equation}
for all $\vC \in \tVHpl$. We refer to \eqref{eq:PDEellLOD} as the \emph{localized multiscale method}. As a next step, we show an estimate for the error $u - \uC^\ell$. 

\begin{theorem}[Error of the localized multiscale method]\label{t:locerrorP}
Let $\ell \in \N$, $f \in H^k(\tri_H)$, $k \in \N_0$, and define ${s := \min\{k,p+1\}}$. Further, let $u \in \Ve$ be the solution of \eqref{eq:PDEellweak} and $\uC^\ell \in \tVHpl$ the solution of \eqref{eq:PDEellLOD}. Then it holds that
\begin{equation}\label{eq:locerrorP}
\|\nabla(u - \uC^\ell)\|_{L^2(D)} 
\lesssim \frac{\Phi(p,s)}{p}\,H^{s+1}\, |f|_{H^s(\tri_H)}
+ \frac{p^3}{H}\,\ell^{(d-1)/2}\exp(-C_\mathrm{dec}\,\ell/p)\,\|f\|_{L^2(D)}
\end{equation}
with the constant $C_\mathrm{dec}$ from Theorem~\ref{t:decay}.
\end{theorem}

\begin{proof}
First, we observe that $\uC^\ell$ is quasi-optimal by the Galerkin orthogonality. Therefore, we obtain 
\begin{equation*}
\|\nabla(u - \uC^\ell)\|_{L^2(D)} \leq\frac{\beta}{\alpha} \inf_{\vC \in \tVHpl} \|\nabla(u - \vC)\|_{L^2(D)}
\leq\frac{\beta}{\alpha}\, \|\nabla(u - \uB^\ell)\|_{L^2(D)},
\end{equation*}
where $\uB^\ell := \calR^\ell \Pi u\in \tVHpl$.
With the triangle inequality and the solution ${\uC \in \tVHp}$ of~\eqref{eq:PDEellOD}, we get that
\begin{equation}\label{eq:triangleineq}
\|\nabla(u - \uB^\ell)\|_{L^2(D)} \leq \|\nabla(u - \uC)\|_{L^2(D)} + \|\nabla(\uC - \uB^\ell)\|_{L^2(D)}.
\end{equation}
The first term can be estimated with Theorem~\ref{t:errODp}, i.e.,
\begin{equation*}
\|\nabla(u - \uC)\|_{L^2(D)} \lesssim \frac{\Phi(p,s)}{p}\,H^{s+1}\, |f|_{H^s(\tri_H)}.
\end{equation*}
For the second term, we set $w:=\uC - \uB^\ell$. Further, for $K \in \tri_H$ we define a~cutoff function ${\eta_K \in W^{1,\infty}(D)}$ with
\begin{equation*}
\begin{aligned}
&0 \leq \eta_K \leq 1,\\
&\eta_K \,= \,0\quad\text{in }\Nb^{\ell-1}(K),\\
&\eta_K\,=\,1\quad\text{in }D\setminus\Nb^{\ell}(K),\\
&\|\nabla\eta_K\|_{L^\infty(D)} \,\leq\, C_\eta\, H^{-1}.
\end{aligned}
\end{equation*} 
We set $R_K := \Nb^{\ell}(K)\setminus\Nb^{\ell-1}(K)$. 
By \eqref{eq:basisfunctionloc} and \eqref{eq:Rlocal}, for each $K \in \tri_H$ there exists a Lagrange multiplier $\lamKl \in \VHp(\Nb^{\ell}(K))$ such that
\begin{equation}\label{eq:RlK}
\begin{aligned}
a(\calR^\ell({\Pi u\vert}_K), v)\qquad\, &+& (\lamKl, v)_{L^2(D)} \quad&=\quad 0,\\
(\calR^\ell({\Pi u\vert}_K),\muH)_{L^2(D)} && &=\quad ({\Pi u\vert}_K, \muH)_{L^2(D)} 
\end{aligned}
\end{equation}
for all $v \in H^1_0(\Nb^{\ell}(K))$ and $\muH \in \VHp(\Nb^{\ell}(K))$.
Noting that 
\begin{equation*} 
(1-\eta_K)w \in H^1_0(\Nb^{\ell}(K))\quad\text{ and }\quad\Pi w = 0, 
\end{equation*}
we obtain with \eqref{eq:saddlepoint} and \eqref{eq:RlK}
\begin{equation}\label{eq:prooflocerrorRpre}
\begin{aligned}
\alpha\,\|\nabla w\|^2_{L^2(D)} &\leq \sum_{K \in \tri_H}a(\calR({\Pi u\vert}_K) - \calR^\ell({\Pi u\vert}_K),w) \\
& = \sum_{K \in \tri_H}-a(\calR^\ell({\Pi u\vert}_K),(1-\eta_K) w + \eta_K w)\\
& = \sum_{K \in \tri_H}\Big((\lamKl,(1-\eta_K) w)_{L^2(R_K)} - a(\calR^\ell({\Pi u\vert}_K),\eta_K w)\Big)\\
&\lesssim \sum_{K \in \tri_H} \Big(\|\lamKl\|_{L^2(R_K)}\,\|w\|_{L^2(R_K)} \\
&\qquad\qquad+ \|\nabla\calR^\ell({\Pi u\vert}_K)\|_{L^2(R_K)}\,\|\nabla (\eta_Kw)\|_{L^2(R_K)}\Big).\\
\end{aligned}
\end{equation}
Next, we bound the terms on the right-hand side of \eqref{eq:prooflocerrorRpre}. 
We observe that by the approximation result \eqref{eq:Piapprox}, we have that
\begin{equation}\label{eq:proofloc1}
\|w\|_{L^2(R_K)} \leq C_{\Pi}\frac{H}{p}\, \|\nabla w\|_{L^2(R_K)}
\end{equation}
and
\begin{equation}\label{eq:proofloc2}
\|\nabla(\eta_K w)\|_{L^2(R_K)} \leq C_{\eta}C_{\Pi}\,p^{-1}\,\|\nabla w\|_{L^2(R_K)} + \|\nabla w\|_{L^2(R_K)}.
\end{equation}
Further, for any $T \in \tri_H$, there exists a bubble function $\rho_T \in W^{1,\infty}(T)\cap H^1_0(T)$ as in the proof of Lemma~\ref{l:infsup} such that
\begin{equation}\label{eq:estLagrange}
\begin{aligned}
\|\lamKl\|^2_{L^2(T)} &\leq 4\, (\lamKl,\rho_T\lamKl)_{L^2(T)}\\
&= -4\, a(\calR^\ell({\Pi u\vert}_K),\rho_T\lamKl)\\
&\lesssim H^{-1}p^2\,\|\nabla\calR^\ell({\Pi u\vert}_K)\|_{L^2(T)}\,\|\lamKl\|_{L^2(T)},
\end{aligned}
\end{equation}
where we employ the estimates \eqref{eq:pol} and \eqref{eq:pol2}.
Using \eqref{eq:proofloc1}-\eqref{eq:estLagrange}, we deduce from \eqref{eq:prooflocerrorRpre} that
\begin{equation}\label{eq:prooflocerrorR}
\alpha\,\|\nabla w\|^2_{L^2(D)} \lesssim \sum_{K \in \tri_H} (p+1)\,\|\nabla\calR^\ell({\Pi u\vert}_K)\|_{L^2(R_K)}\,\|\nabla w\|_{L^2(R_K)}.
\end{equation}
We now use Theorem~\ref{t:decay} and Remark~\ref{r:decay}, which both equivalently hold with $\calR$ replaced by $\calR^\ell$. Therefore, we get with \eqref{eq:prooflocerrorR} and \eqref{eq:decay} that
\begin{equation*}
\begin{aligned}
\|\nabla w\|^2_{L^2(D)} &\lesssim \sum_{K \in \tri_H} (p+1)\,\|\nabla\calR^\ell({\Pi u\vert}_K)\|_{L^2(D\setminus\Nb^{\ell-1}(K))}\,\|\nabla w\|_{L^2(R_K)}\\
&\lesssim \frac{p^3}{H}\,\exp(-C_\mathrm{dec}\,\ell/p)\Big(\sum_{K \in \tri_H}\|{\Pi u\vert}_K\|^2_{L^2(K)}\Big)^{1/2}
\Big(\sum_{K \in \tri_H}\|\nabla w\|^2_{L^2(R_K)}\Big)^{1/2}\\
&\lesssim \frac{p^3}{H}\,\ell^{(d-1)/2}\,\exp(-C_\mathrm{dec}\,\ell/p)\,\|\Pi u\|_{L^2(D)}\,
\|\nabla w\|_{L^2(D)},
\end{aligned}
\end{equation*}
where we also employ the discrete Cauchy-Schwarz inequality and the stability of \eqref{eq:RlK}, i.e.,
\begin{equation*}
\|\nabla\calR^\ell({\Pi u\vert}_K)\|_{L^2(D)} \lesssim \frac{p^2}{H}\,\|{\Pi u\vert}_K\|_{L^2(K)}
\end{equation*}
for any $K \in \tri_H$. 
We now go back to \eqref{eq:triangleineq} and finally obtain
\begin{equation*}
\begin{aligned}
\|\nabla(u - \uC^\ell)\|_{L^2(D)} 
&\lesssim \frac{\Phi(p,s)}{p}\,H^{s+1}\, |f|_{H^s(\tri_H)}
+ \frac{p^3}{H}\,\ell^{(d-1)/2}\exp(-C_\mathrm{dec}\,\ell/p)\,\|\Pi u\|_{L^2(D)}\\
&\lesssim \frac{\Phi(p,s)}{p}\,H^{s+1}\, |f|_{H^s(\tri_H)}
+ \frac{p^3}{H}\,\ell^{(d-1)/2}\exp(-C_\mathrm{dec}\,\ell/p)\,\|f\|_{L^2(D)},
\end{aligned}
\end{equation*}
where we use the stability of $\Pi$ and \eqref{eq:stabuell}.
\end{proof}

\begin{remark}
The additional $H$ in the denominator of the estimate in Theorem~\ref{t:locerrorP} may be explained by the fact that the localization error $\uC -\uB^\ell$ is measured in the $H^1$-norm and is bounded by $\Pi u$, which is measured in the $L^2$-norm. Although this seems suboptimal, the pollution in terms of $H$ in the second term of \eqref{eq:locerrorP} is also observed in our numerical experiments; see Section~\ref{s:numexpp}. 
\end{remark}

We can now use Theorem~\ref{t:locerrorP} to quantify the choice of the localization parameter $\ell$ with respect to the polynomial degree $p$ and the mesh size $H$ dependent on the regularity of the right-hand side $f$.

\begin{corollary}\label{cor:errorLODR} 
Let $f \in H^k(\tri_H),\, k \in \N_0$, and define $s := \min\{k,p+1\}$. Further, let $u \in \Ve$ be the solution of \eqref{eq:PDEellweak}, and $\uC^\ell \in \tVHpl$ the solution of \eqref{eq:PDEellLOD}. Then, for 
\begin{equation}\label{eq:lscale}
\ell \gtrsim |\log H|\,p\,(s+1) + (\log p) \,p\,(s+1),
\end{equation}
it holds that
\begin{equation*}
\|\nabla(u - \uC^\ell)\|_{L^2(D)} \lesssim \frac{\Phi(p,s)}{p}\,H^{s+1}\, |f|_{H^s(\tri_H)}
+ \Big(\frac{H}{p}\Big)^{s+1}\,\|f\|_{L^2(D)}.
\end{equation*}
\end{corollary}

Note that if $k = 0$ and $p = 1$, Corollary~\ref{cor:errorLODR} provides a similar error estimate as for the conforming first-order LOD method with the same scaling of $\ell$; see, e.g., \cite{HenP13}. Of course, if we increase $p$, the localization parameter $\ell$ in Theorem~\ref{t:locerrorP} needs to grow as well in order to maintain the high convergence rate of Theorem~\ref{t:errODp} with respect to $H$ and $p$. Nevertheless, the experiments in Section~\ref{s:numexpp} indicate that the $p$-dependence of $\ell$ in \eqref{eq:lscale} might be too pessimistic and the decay property of Theorem~\ref{t:decay} in fact even slightly improves for larger values of $p$. Before we turn our attention to these numerical investigations of the high-order method, we first need to discuss the last step towards a fully practical method, i.e., the discretization at the microscopic scale.

\subsection{Microscopic discretization}\label{ss:micdisc}

The localized operator $\calR^\ell$ does still not provide a~fully discrete method since the localized basis functions \eqref{eq:basisfunctionloc} are obtained by solving infinite-dimensional auxiliary problems. The easiest approach to resolve this issue is to introduce a (conforming) fine finite element space $\Vhp \subseteq \Ve$ based on a~decomposition $\tri_h$ with mesh parameter $h$ and polynomial degree $p^\prime$ that replaces the space $\Ve$ in the above construction.
Ideally, the classical Galerkin solution in $\Vhp$ should fulfill an estimate similar to the one in Theorem~\ref{t:errODp}. Motivated by error estimates of high-order finite element methods (see, e.g., \cite{BabG96,Sch98}), for $f \in H^k(D),\, k \in \N_0$, we assume that 
\begin{equation}\label{eq:approxfinep}
\|\nabla(u - \uh)\|_{L^2(D)} \lesssim\frac{\Phi(p^\prime,s^\prime)}{p^\prime}\,(C_\epsilon\, h)^{s^\prime+1}\, |f|_{H^{s^\prime}(D)}, 
\end{equation}
where $u\in \Ve$ is the solution of \eqref{eq:PDEellweak}, $s^\prime := \min\{k,p^\prime+1\}$, and $\uh \in \Vhp$ is the solution of 
\begin{equation}\label{eq:finesolp}
a(\uh,\vh) = (f,\vh)_{L^2(D)}
\end{equation}
for all $\vh \in \Vhp$. Note that the right-hand side of \eqref{eq:approxfinep} depends on the fine-scale parameter $\epsilon$ through the constant $C_\epsilon$. This is typical for classical finite element spaces which do not take into account microscopic information.  

We emphasize that on the one hand, the ideal approximation $\uC \in \tVHp$ characterized by \eqref{eq:PDEellOD} fulfills the high-order estimate quantified in Theorem~\ref{t:errODp} by the (piecewise) regularity of the right-hand side $f$ only. On the other hand, in order to obtain a high-order estimate of the form \eqref{eq:approxfinep} for the classical finite element space $\Vhp$, the regularity of $f$ needs to hold globally. Further, one requires additional smoothness assumptions on the domain $D$ as well as on the coefficient $A$ (see, e.g., \cite[Thm.~5 in Sec.~6.3]{Eva10}) and, in particular, the microscopic scale $\epsilon$ needs to be resolved. 
Another problem that occurs when discretizing the fine scales is the fact that the proof of the inf-sup condition in Lemma~\ref{l:infsup} is explicitly based on the space $\Ve$. The result does not directly follow for subspaces of $\Ve$ and a similar inf-sup condition needs to be proved for the respective discrete space $\Vhp$ at hand. 

With these problems in mind, the following lemma provides a condition on the fine mesh parameter $h$ for which the inf-sup condition \eqref{eq:infsupL2} and thus the surjectivity results in Theorem~\ref{t:surjectivePi} and Corollary~\ref{c:locbubble} remain valid if $\Ve$ is replaced by the first-order space $\Vh \subseteq \Ve$, for which we omit the subscript $1$. The explicit choice of the polynomial degree $p^\prime = 1$ is motivated by the fact that high-order estimates for the classical conforming finite element space $\Vhp$ would require additional smoothness assumptions as mentioned above. 

\begin{lemma}[Discrete local inf-sup condition]\label{l:infsuph}
Let $K \in \tri_H$. Then there exists a constant $C > 0$ independent of $h$, $H$, and $p$ such that for 
\begin{equation*}
h \leq C\,Hp^{-2}
\end{equation*}
the inf-sup condition
\begin{equation}\label{eq:infsupL2h}
\adjustlimits\inf_{q \in \VHp(K)}\sup_{\vh \in \Vh \cap H^1_0(K)} \frac{(q,\vh)_{L^2(K)}}{\|q\|_{L^2(K)}\,\|\nabla \vh\|_{L^2(K)}} \geq \gamma_h > 0
\end{equation}
holds with $\gamma_h \sim Hp^{-2}$. 
\end{lemma}

\begin{proof}
As in the proof of Lemma~\ref{l:infsup}, let $\kappa \subseteq K$ be such that its edges and faces are parallel to the ones of $K$, ${\dist(\kappa,\partial K) = C_{\dist}\,Hp^{-2}}$, and
\begin{equation}\label{eq:polh}
\| q \|^2_{L^2(\kappa)} \geq \frac{1}{4}\| q \|^2_{L^2(K)}
\end{equation}
for all $q \in \VHp(K)$. 
Now, let ${\rho \in W^{1,\infty}(K)\cap H^1_0(K)}$ with 
\begin{equation*}
\begin{aligned}
&0 \,\leq \,\rho\,\leq\,1,\\
&\rho\,=\,1\quad\text{in }\kappa,\\
&\|\nabla\rho\|_{L^\infty(K)} \,\leq \,C_{\rho}\,H^{-1}p^2,
\end{aligned}
\end{equation*}
where $C_{\rho}$ depends on $C_{\dist}$.
Next, we define, for any $q \in \VHp(K)$, the auxiliary function \mbox{$w_q \in \Vh \cap H^1_0(K)$} as the solution of
\begin{equation}\label{eq:auxprob}
(w_q,\vh)_{L^2(K)} = (q,\vh)_{L^2(K)}
\end{equation}
for all $\vh \in \Vh \cap H^1_0(K)$. Note that $w_q$ is unique by the inverse inequality
\begin{equation}\label{eq:auxinv}
\|\nabla \vh\|_{L^2(K)} \leq \tilde C_{\mathrm{inv}} \,h^{-1}\,\|\vh\|_{L^2(K)} 
\end{equation}
and the Lax-Milgram Theorem. 
The last auxiliary ingredient is an estimate of the form
\begin{equation*}
\|q\|_{L^2(K)} \lesssim \|w_q\|_{L^2(K)}
\end{equation*}
which can be obtained using a projection operator ${\IhK\colon L^2(K) \to \Vh \cap H^1_0(K)}$ which fulfills the following stability and approximation properties. For all $v \in L^2(K)$, it holds that
\begin{equation}\label{eq:stabilityIhK}
\|\IhK v\|_{L^2(K)} \leq C_{\IhK} \|v\|_{L^2(K)}
\end{equation}
and, for any $v \in H_0^1(K)$, we have
\begin{equation}\label{eq:interpolationIhK}
\|h^{-1}(v-\IhK v)\|_{L^2(K)} + \|\nabla\IhK v\|_{L^2(K)} \leq C_{\IhK} \,\|\nabla v\|_{L^2(K)}.
\end{equation}
For an explicit choice of $\IhK$, we refer to \cite{Osw93,Bre94,ErnG17}. With \eqref{eq:polh}, \eqref{eq:auxprob}, \eqref{eq:stabilityIhK}, \eqref{eq:interpolationIhK}, and \eqref{eq:pol2}, we can show that 
\begin{equation*}
\begin{aligned}
\tfrac{1}{4}\|q\|_{L^2(K)}^2 &\leq \|q\|_{L^2(\kappa)}^2 \leq (q, \rho q)_{L^2(K)} = (q, \IhK(\rho q))_{L^2(K)} + (q,(\opid - \IhK)(\rho q))_{L^2(K)}\\
& = (w_q, \IhK(\rho q))_{L^2(K)} + (q,(\opid - \IhK)(\rho q))_{L^2(K)}\\
& \leq \|w_q\|_{L^2(K)}\, C_{\IhK}\,\|\rho q\|_{L^2(K)} + \|q\|_{L^2(K)}\,C_{\IhK}h\,\|\nabla(\rho q)\|_{L^2(K)}\\
&\leq C_{\IhK}\,\|w_q\|_{L^2(K)}\,\|q\|_{L^2(K)}  + C_{\IhK}(C_\rho + C_\mathrm{inv})\,hH^{-1}p^2\,\|q\|_{L^2(K)}^2.
\end{aligned}
\end{equation*}
Assuming that
\begin{equation*}
C_{\IhK}(C_\rho + C_\mathrm{inv})\,hH^{-1}p^2 \leq \frac{1}{8},
\end{equation*} 
we thus obtain
\begin{equation}\label{eq:auxqwq}
\|q\|_{L^2(K)} \leq 8\,C_{\IhK}\,\|w_q\|_{L^2(K)}.
\end{equation}
Finally, with the estimates \eqref{eq:auxqwq} and \eqref{eq:auxinv}, it holds that
\begin{equation*}
\begin{aligned}
\adjustlimits\inf_{q \in \VHp(K)}\sup_{\vh \in \Vh\cap H^1_0(K)} \frac{(q,\vh)_{L^2(K)}}{\|q\|_{L^2(K)}\,\|\nabla \vh\|_{L^2(K)}} &\geq 
\inf_{q \in \VHp(K)}\frac{(w_q,w_q)_{L^2(K)}}{\|q\|_{L^2(K)}\,\|\nabla w_q\|_{L^2(K)}}\\
&\geq \inf_{q \in \VHp(K)}\frac{1}{8\,C_{\IhK}}\frac{\|w_q\|^2_{L^2(K)}}{\|w_q\|_{L^2(K)}\,\|\nabla w_q\|_{L^2(K)}}\\
&\geq \frac{h}{8\,C_{\IhK}\tilde C_{\mathrm{inv}}} =: \gamma_h > 0.
\end{aligned}
\end{equation*}
For $h \sim Hp^{-2}$, this is the assertion. For $h \lesssim Hp^{-2}$, there exists an auxiliary $h^\prime \sim Hp^{-2}$ such that $V_{h^\prime} \subseteq \Vh$ and thus 
\begin{align*}
\adjustlimits\inf_{q \in \VHp(K)}\sup_{\vh \in \Vh\cap H^1_0(K)} \frac{(q,\vh)_{L^2(K)}}{\|q\|_{L^2(K)}\,\|\nabla \vh\|_{L^2(K)}} 
&\geq \adjustlimits\inf_{q \in \VHp(K)}\sup_{v \in V_{h^\prime}\cap H^1_0(K)} \frac{(q,v)_{L^2(K)}}{\|q\|_{L^2(K)}\,\|\nabla v\|_{L^2(K)}} \\
&\geq \gamma_{h^\prime} \sim Hp^{-2}. \qedhere
\end{align*}
\end{proof}

With Lemma~\ref{l:infsuph}, we can replace $\Ve$ (and the solution $u\in\Ve$ of \eqref{eq:PDEellweak}) in the construction of Sections~\ref{s:highorder} and~\ref{s:practical} by a conforming $Q_1$ finite element space $\Vh$ (and the classical Galerkin approximation $\uh\in \Vh$) provided that $h$ is sufficiently small with respect to $H$ and $p$ and, additionally, resolves the microscopic information on the scale $\epsilon$. 
This is quantified with the resolution conditions
\begin{equation}\label{eq:resolutionp}
C_\epsilon\, h \lesssim\frac{\Phi(p,s)}{p}\,H^{s+1} \quad\text{ and }\quad h \lesssim Hp^{-2},
\end{equation}
where the constant $C_\epsilon$ indicates the dependence on the microscopic scale $\epsilon$ as in \eqref{eq:approxfinep}. While a resolution condition on $h$ with respect to $H$ and $p$ of the form $h \leq Hp^{-s}$ for some $s \geq 1$ seems natural to resolve high-order functions, the left condition in \eqref{eq:resolutionp} is mainly motivated by the aim to retain the convergence properties with respect to $H$ and $p$ as derived in the previous subsections. In a~more practical manner, one could alternatively prescribe some certain tolerance and balance $h$, $p^\prime$, $H$, and $p$ such that the given tolerance is reached with the respective approximation. 

Note that a discrete inf-sup condition as in Lemma~\ref{l:infsuph} may also be obtained for a high-order conforming finite element space and relaxes the resolution condition \mbox{$h \lesssim H p^{-2}$} dependent on the choice of $p^\prime$. If additional smoothness conditions hold, the use of a high-order space can further provide a relaxation of the left resolution condition in \eqref{eq:resolutionp} on $h$ if $p^\prime$ is suitably coupled to $h$ and $\epsilon$ in the spirit of \cite[Cor.~5.3]{PetS12}.

Although high-order constructions may generally be considered for the fine discretization, we remain with the first-order setting where $p^\prime = 1$, which only requires minimal regularity assumptions.
We introduce the additional parameter $h$ in the above construction if $\Ve$ is replaced by $\Vh$, i.e., we write 
\begin{equation*}
\calR_h,\,\calR_h^\ell,\,\tVHhp,\,\tVHhpl\quad\text{ instead of }\quad\calR,\,\calR^\ell,\,\tVHp,\,\tVHpl.
\end{equation*} 
Further, the solution $\uhC \in \tVHhpl$ of the \emph{fully discrete multiscale method} is determined by
\begin{equation}\label{eq:PDEellLODh}
a(\uhC^\ell,\vhC) = (f,\vhC)_{L^2(D)}
\end{equation}
for all $\vhC \in \tVHhpl$. The error of the fully discrete approach is quantified in the next theorem.

\begin{theorem}[Error of the fully discrete multiscale method]\label{t:errfullydiscretep}
Assume $f \in H^k(\tri_H)$, ${k \in \N_0}$, and let $s := \min\{k,p+1\}$. Further, suppose that the resolution conditions \eqref{eq:resolutionp} hold and let $u \in \Ve$ be the solution of \eqref{eq:PDEellweak} and $\uhC^\ell \in \tVHhpl$ the solution of \eqref{eq:PDEellLODh}. Then, with the choice
\begin{equation*}
\ell \gtrsim |\log H|\,p\,(s+1) + (\log p) \,p\,(s+1),
\end{equation*}
it holds that
\begin{equation*}
\|\nabla(u - \uhC^\ell)\|_{L^2(D)} \lesssim \frac{\Phi(p,s)}{p}\,H^{s+1}\,\big(\|f\|_{L^2(D)} +  |f|_{H^s(\tri_H)}\big)
+ \Big(\frac{H}{p}\Big)^{s+1}\,\|f\|_{L^2(D)}.
\end{equation*}
\end{theorem}

\begin{proof}
The assertion follows from a simple triangle inequality, the estimate \eqref{eq:approxfinep} with $s^\prime = 0$, the resolution conditions \eqref{eq:resolutionp}, and Corollary~\ref{cor:errorLODR} in the case where $\Ve$ is replaced by $\Vh$.
To be more precise, with the solution $\uh \in \Vh$ of \eqref{eq:finesolp}, we obtain
\begin{align*}
\|\nabla(u - \uhC^\ell)\|_{L^2(D)} &\leq \|\nabla(u - \uh)\|_{L^2(D)} + \|\nabla(\uh - \uhC^\ell)\|_{L^2(D)}\\
& \lesssim C_\epsilon \,h\, \|f\|_{L^2(D)} + \frac{\Phi(p,s)}{p}\,H^{s+1}\, |f|_{H^s(\tri_H)}
+ \Big(\frac{H}{p}\Big)^{s+1}\,\|f\|_{L^2(D)}\\
&\lesssim \frac{\Phi(p,s)}{p}\,H^{s+1}\,\big(\|f\|_{L^2(D)} +  |f|_{H^s(\tri_H)}\big)
+ \Big(\frac{H}{p}\Big)^{s+1}\,\|f\|_{L^2(D)}. \qedhere
\end{align*}
\end{proof}

\section{Numerical Experiments}\label{s:numexpp}

In this section, we present some examples to investigate the results of the previous sections. We remark that if the exact solution $u \in \Ve$ of \eqref{eq:PDEellweak} is not explicitly given, only the errors between the discrete solutions $\uh \in \Vh$ of \eqref{eq:finesolp} and $\uhC^\ell \in \tVHhpl$ of \eqref{eq:PDEellLODh} can be measured. Thus, we need to pose the assumption that the chosen mesh parameter~$h$ is indeed small enough as quantified in Section~\ref{ss:micdisc}, and use $\uh$ as the reference solution. In general, any other fine discretization could be used to obtain a reference solution such as, e.g., a discontinuous Galerkin approach as in \cite{ElfGMP13,ElfGM13}. We measure all occurring errors in the energy norm $\|\cdot\|_a := \|A^{1/2}\nabla\cdot\|_{L^2(D)}$.

\begin{figure}
	\centering
	\begin{subfigure}[b]{0.48\textwidth}
		\centering
		\includegraphics[width=1.13\textwidth]{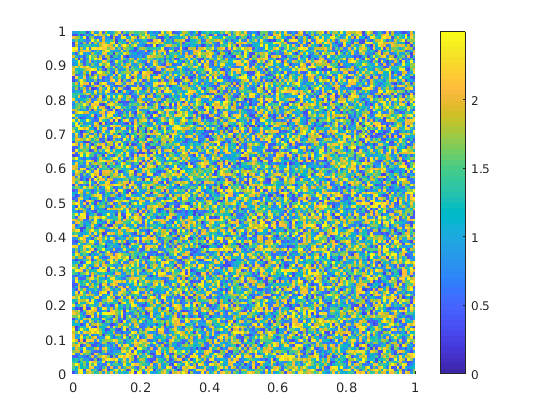}  
	\end{subfigure}
	\quad
	\begin{subfigure}[b]{0.48\textwidth}  
		\centering 
		\includegraphics[width=1.13\textwidth]{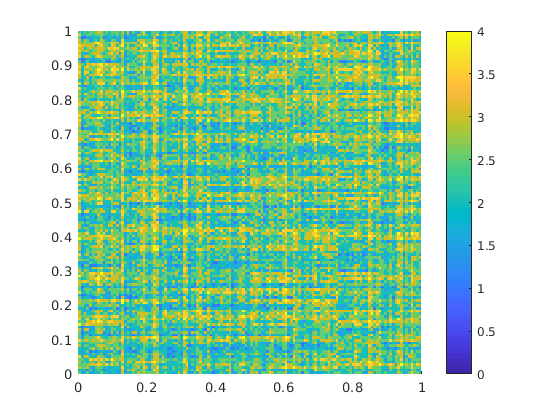}
	\end{subfigure}
	\caption[Multiscale coefficients]%
	{\small Multiscale coefficients $A_1$ (left) and $A_2$ (right) on the scale $\epsilon = 2^{-7}$.} 
	\label{fig:coeffA12}
\end{figure}
\begin{figure}
	\begin{center}
		\scalebox{0.84}{
			\begin{tikzpicture}
			
			\begin{axis}[%
			width=2.8in,
			height=2.5in,
			at={(0.745in,0.481in)},
			scale only axis,
			xmode=log,
			xmin=0.026,
			xmax=0.6,
			xlabel={\large{mesh size $H$}},
			xminorticks=true,
			ymode=log,
			ymin=7e-06,
			ymax=2,
			yminorticks=true,
			axis background/.style={fill=white},
			]
			\addplot [color=myBlue, mark=square, line width=1.5pt, mark size=2.0pt]
			table[row sep=crcr]{%
				5.000000000000000e-01     9.730667022671702e-01\\
				2.500000000000000e-01     2.075550050014089e-01\\
				1.250000000000000e-01     3.745251816543015e-02\\
				6.250000000000000e-02     2.070388614131157e-02\\
				3.125000000000000e-02     2.437237089796834e-02\\
			};
			
			\addplot [color=myBlue, mark=square, line width=1.5pt, mark size=4.0pt]
			table[row sep=crcr]{%
				5.000000000000000e-01     9.730667022671702e-01\\
				2.500000000000000e-01     2.075550050014089e-01\\
				1.250000000000000e-01     3.875464986983049e-02\\
				6.250000000000000e-02     8.230597251245588e-03\\
				3.125000000000000e-02     3.840931299971967e-03\\
			};
			
			\addplot [color=myRed, mark=o, line width=1.5pt, mark size=2.0pt]
			table[row sep=crcr]{%
				5.000000000000000e-01     4.806091090019082e-01\\
				2.500000000000000e-01     7.358267468701588e-02\\
				1.250000000000000e-01     4.812480180129881e-03\\
				6.250000000000000e-02     1.139518836274180e-03\\
				3.125000000000000e-02     1.602575083216002e-03\\
			};
			
			\addplot [color=myRed, mark=o, line width=1.5pt, mark size=4.0pt]
			table[row sep=crcr]{%
				5.000000000000000e-01     4.806091090019082e-01\\
				2.500000000000000e-01     7.358267468701588e-02\\
				1.250000000000000e-01     4.737689765928027e-03\\
				6.250000000000000e-02     2.900232887639292e-04\\
				3.125000000000000e-02     2.523957054107803e-04\\
			};
			
			\addplot [color=myGreen, mark=triangle, line width=1.5pt, mark size=2.0pt]
			table[row sep=crcr]{%
				5.000000000000000e-01     2.499572077128505e-01\\
				2.500000000000000e-01     1.029278503370555e-02\\
				1.250000000000000e-01     4.092926314435829e-04\\
				6.250000000000000e-02     4.671194628241819e-05\\
				3.125000000000000e-02     1.373507945272820e-04\\
			};
			
			\addplot [color=myGreen, mark=triangle, line width=1.5pt, mark size=4.0pt]
			table[row sep=crcr]{%
				5.000000000000000e-01     2.499572077128505e-01\\
				2.500000000000000e-01     1.029278503370555e-02\\
				1.250000000000000e-01     4.134678668024592e-04\\
				6.250000000000000e-02     1.773727064501356e-05\\
				3.125000000000000e-02     1.728678519602808e-05\\
			};
			
			\addplot [color=myOrange, mark=diamond, line width=1.5pt, mark size=4.0pt]
			table[row sep=crcr]{%
				5.000000000000000e-01     9.999992533246688e-01\\
				2.500000000000000e-01     9.721261715183733e-01\\
				1.250000000000000e-01     4.453894480390736e-01\\
				6.250000000000000e-02     1.262935055352588e-01\\
				3.125000000000000e-02     3.463332999441188e-02\\
			};
			
			\addplot [color=color3, mark=pentagon, line width=1.5pt, mark size=4.0pt]
			table[row sep=crcr]{%
				5.000000000000000e-01     1.000000000000003e+00\\
				2.500000000000000e-01     1.104750257329905e+00\\
				1.250000000000000e-01     6.669155853265974e-01\\
				6.250000000000000e-02     4.466144022629336e-01\\
				3.125000000000000e-02     3.711601189863007e-01\\
			};
			
			\addplot [color=black, dotted,line width=1.5pt]
			table[row sep=crcr]{
				0.5	0.625\\
				0.028 0.00010976\\
			};
			
			\addplot [color=black, dotted,line width=1.5pt,forget plot]
			table[row sep=crcr]{
				0.5	0.3125\\
				0.037 9.370804999999997e-06\\
			};
			\addplot [color=black, dotted,line width=1.5pt,forget plot]
			table[row sep=crcr]{
				0.5	0.15625\\
				0.073 1.0365357964999998e-05\\
			};
			\addplot [color=black, dotted,line width=1.5pt,forget plot]
			table[row sep=crcr]{
				0.5	1.25\\
				0.028 0.00392\\
			};
			\end{axis}
			
			\begin{axis}[%
			width=2.8in,
			height=2.5in,
			at={(4.4in,0.481in)},
			scale only axis,
			xmode=log,
			xmin=0.026,
			xmax=0.6,
			xlabel={\large{mesh size $H$}},
			xminorticks=true,
			ymode=log,
			ymin=6e-06,
			ymax=1.5,
			yminorticks=true,
			ylabel={\large{rel err in $\|\cdot\|_{a}$}},
			axis background/.style={fill=white},
			legend style={at={(-1.3,-.4)}, anchor=south west, legend columns=5, legend cell align=left, align=left, draw=white!15!black},
			]
			
			\addplot [color=myBlue, mark=square, line width=1.5pt, mark size=2.0pt]
			table[row sep=crcr]{
				0.500000000000000   0.183210941411053\\
				0.250000000000000   0.033294677578032\\
				0.125000000000000   0.022991075136924\\
				0.062500000000000   0.062568704725411\\
				0.031250000000000   0.139685347605383\\
			};
			\addlegendentry{$p=1$, $\ell=3\;\;$}
			
			\addplot [color=myBlue, mark=square, line width=1.5pt, mark size=4.0pt]
			table[row sep=crcr]{%
				0.500000000000000   0.183210941411053\\
				0.250000000000000   0.033294677578032\\
				0.125000000000000   0.006933137644109\\
				0.062500000000000   0.010033865713111\\
				0.031250000000000   0.023507930653259\\
			};
			\addlegendentry{$p=1$, $\ell=4\;\;$}
			
			\addplot [color=myRed, mark=o, line width=1.5pt, mark size=2.0pt]
			table[row sep=crcr]{%
				0.500000000000000   0.057732909042062\\
				0.250000000000000   0.005371676064840\\
				0.125000000000000   0.001238818334371\\
				0.062500000000000   0.003950387855762\\
				0.031250000000000   0.010875812596435\\
			};
			\addlegendentry{$p=2$, $\ell=4\;\;$}
			
			\addplot [color=myRed, mark=o, line width=1.5pt, mark size=4.0pt]
			table[row sep=crcr]{%
				0.500000000000000   0.057732909042062\\
				0.250000000000000   0.005371676064840\\
				0.125000000000000   0.000358018016055\\
				0.062500000000000   0.000463359945590\\
				0.031250000000000   0.001370249726825\\
			};
			\addlegendentry{$p=2$, $\ell=5\;\;$}
			
			\addplot [color=myGreen, mark=triangle, line width=1.5pt, mark size=2.0pt]
			table[row sep=crcr]{%
				0.500000000000000   0.014529119651863\\
				0.250000000000000   0.000580677349462\\
				0.125000000000000   0.000059127104594\\
				0.062500000000000   0.000191421201561\\
				0.031250000000000   0.000835248134753\\
			};
			\addlegendentry{$p=3$, $\ell=5\;\;$}
			
			\addplot [color=myGreen, mark=triangle, line width=1.5pt, mark size=4.0pt]
			table[row sep=crcr]{%
				0.500000000000000   0.014529119651863\\
				0.250000000000000   0.000580677349462\\
				0.125000000000000   0.000023021620692\\
				0.062500000000000   0.000017690927936\\
				0.031250000000000   0.000090135487256\\
			};
			\addlegendentry{$p=3$, $\ell=6\;\;$}
			
			\addplot [color=myOrange, mark=diamond, line width=1.5pt, mark size=4.0pt]
			table[row sep=crcr]{%
				5.000000000000000e-01    0.8559\\ 
				2.500000000000000e-01     0.4498\\ 
				1.250000000000000e-01     0.1550\\
				6.250000000000000e-02     0.0477\\ 
				3.125000000000000e-02     0.0158\\ 
			};
			\addlegendentry{LOD,\,\,\,$\ell = 2\;\;$}
			
			\addplot [color=color3, mark=pentagon, line width=1.5pt, mark size=4.0pt]
			table[row sep=crcr]{%
				5.000000000000000e-01    0.9046\\ 
				2.500000000000000e-01     0.5857\\ 
				1.250000000000000e-01     0.3389\\ 
				6.250000000000000e-02     0.2165\\ 
				3.125000000000000e-02     0.1639\\ 
			};
			\addlegendentry{$Q_1$ FEM$\;\;$}

			\addplot [color=black, dotted,line width=1.5pt]
			table[row sep=crcr]{
				0.5	0.1\\
				0.022	8.5184e-06\\
			};
			\addlegendentry{order $2,3,4,5$}
			
			\addplot [color=black, dotted,line width=1.5pt,forget plot]
			table[row sep=crcr]{
				0.5	0.05\\
				0.0572	8.563949588480001e-06\\
			};
			\addplot [color=black, dotted,line width=1.5pt,forget plot]
			table[row sep=crcr]{
				0.5	0.025\\
				0.1	8e-06\\
			};
			\addplot [color=black, dotted,line width=1.5pt,forget plot]
			table[row sep=crcr]{
				0.5	0.2\\
				0.0135 0.0001458\\
			};
			\end{axis}
			\end{tikzpicture}%
		}
	\end{center}
	\caption[Errors of the high-order multiscale method w.r.t~$H$]
	{\small Errors of the high-order multiscale method for different values of~$\ell$ and~$p$, the LOD method, and the classical conforming $Q_1$ finite element method in the relative energy norm for the first (left) and the second model (right) with respect to $H$.} 
	\label{fig:errors2d}
\end{figure}
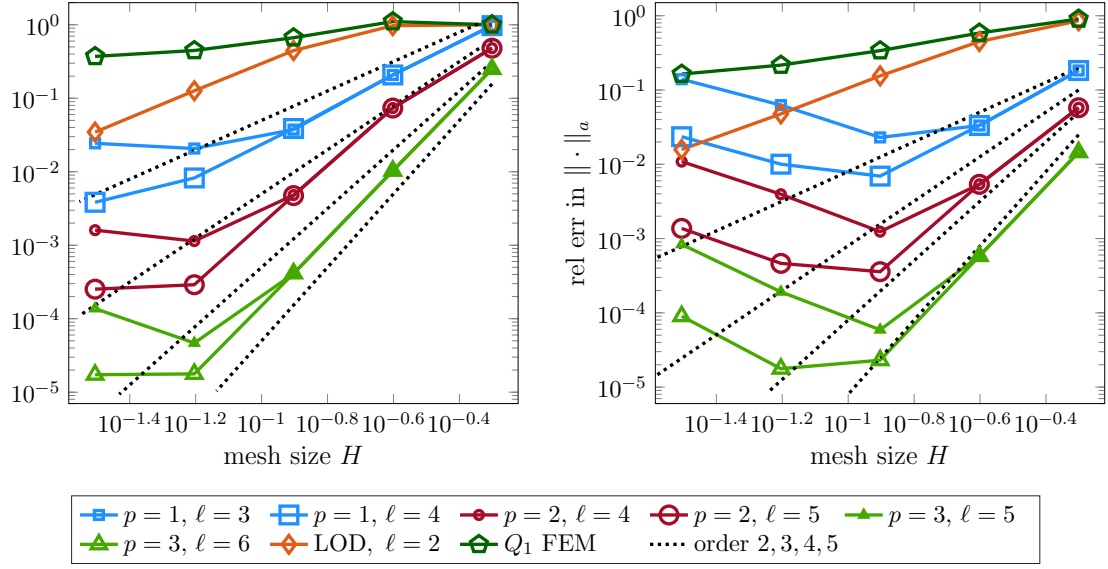
We consider the domain $D = (0,1)^2$ as well as the two scalar diffusion coefficients $A_1$ and $A_2$ as depicted in Figure~\ref{fig:coeffA12}. These coefficients are piecewise constant on a mesh $\tri_\epsilon$ with mesh parameter $\epsilon = 2^{-7}$. 
The coefficients $A_1$ and $A_2$ take values in $[0.25,2.5]$ and~$[1,4]$, respectively. 
Further, we take the right-hand sides
\begin{equation*}
f_1(x) = \sin(5\pi\,x_1)\cos(3\pi\,x_2)
\end{equation*}
and 
\begin{equation*}
f_2(x) = (x_1+\cos(3\pi\,x_1))\,x_2^3.
\end{equation*}
\noindent For the first model, we choose the coefficient $A=A_1$ and the right-hand side ${f=f_1}$ in \eqref{eq:PDEellweak} and compute the solution $\uhC^\ell \in \tVHhpl$ of \eqref{eq:PDEellLODh} for multiple choices of the polynomial degree $p$ and the localization parameter $\ell$. The relative energy errors of these approximations with respect to the reference solution on the scale $h = 2^{-9}$ are depicted in Figure~\ref{fig:errors2d}~(left).
Similarly, we present the energy errors for the second model with the coefficient $A_2$ and the right-hand side $f_2$ in Figure~\ref{fig:errors2d}~(right), where again $h = 2^{-9}$. 
The error curves in both examples show a convergence rate between $p+1$ and $p+2$ with respect to $H$ for different polynomial degrees $p$ if $\ell$ is chosen large enough. These results are in line with the findings in Theorem~\ref{t:locerrorP} which predicts a convergence rate of up to order $p+2$ in $H$ dependent on the regularity of $f$ and provided that the second term in the estimate \eqref{eq:locerrorP} is small enough. 
Apart from the observed high-order rates for appropriate parameter regimes, the two examples also indicate that there might be a pollution in terms of some negative power of $H$ as obtained from the theory. That is, instead of a stagnation of the error curve for smaller $H$, the overall error grows again if $\ell$ is not chosen appropriately. 

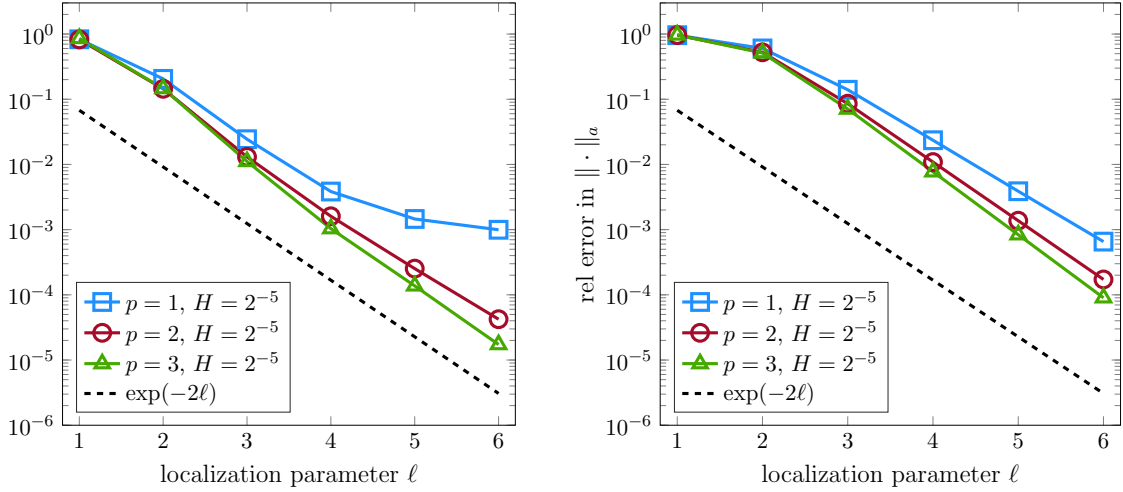
\begin{figure}
	\begin{center}
		\scalebox{0.79}{
			\begin{tikzpicture}
			
			\begin{axis}[%
			width=3.0in,
			height=2.8in,
			scale only axis,
			at={(0.758in,0.481in)},
			scale only axis,
			xmin=0.8,
			xlabel={\large{localization parameter $\ell$}},
			xmax=6.2,
			ymode=log,
			ymin=1e-06,
			ymax=3,
			yminorticks=true,
			axis background/.style={fill=white},
			legend style={legend cell align=left, align=left,
				draw=white!15!black},
			legend pos = south west
			]
			
			\addplot [color=myBlue, mark=square, line width=1.5pt, mark size=4.0pt]
			table[row sep=crcr]{%
				1.000000000000000e+00     8.312655750806304e-01\\
				2.000000000000000e+00     2.052601656086005e-01\\
				3.000000000000000e+00     2.437237089796834e-02\\
				4.000000000000000e+00     3.840931299971967e-03\\
				5.000000000000000e+00     1.457110494435946e-03\\
				6.000000000000000e+00     9.952132063468281e-04\\
			};
			\addlegendentry{$p=1$, $H=2^{-5}$}
			
			\addplot [color=myRed, mark=o, line width=1.5pt, mark size=4.0pt]
			table[row sep=crcr]{%
				1.000000000000000e+00     8.253732102133096e-01\\
				2.000000000000000e+00     1.445456531724117e-01\\
				3.000000000000000e+00     1.306743808447378e-02\\
				4.000000000000000e+00     1.602575083216002e-03\\
				5.000000000000000e+00     2.523957054107803e-04\\
				6.000000000000000e+00     4.226536086813621e-05\\
			};
			\addlegendentry{$p=2$, $H=2^{-5}$}
			
			\addplot [color=myGreen, mark=triangle, line width=1.5pt, mark size=4.0pt]
			table[row sep=crcr]{%
				1.000000000000000e+00     8.549519723339252e-01\\
				2.000000000000000e+00     1.458432388782582e-01\\
				3.000000000000000e+00     1.110233495971014e-02\\
				4.000000000000000e+00     1.035703365792050e-03\\
				5.000000000000000e+00     1.373507945272820e-04\\
				6.000000000000000e+00     1.728678519602808e-05\\
			};
			\addlegendentry{$p=3$, $H=2^{-5}$}
			
			\addplot [color=black, dashed,line width=1.5pt]
			table[row sep=crcr]{%
				1	0.0676676416183064\\
				2	0.00915781944436709\\
				3	0.00123937608833318\\
				4	0.000167731313951256\\
				5	2.26999648812424e-05\\
				6	3.0721061766641e-06\\
			};
			\addlegendentry{$\exp(-2\ell)$}
			
			\end{axis}
			\end{tikzpicture}%
		}
		\hfill
		\scalebox{0.79}{
			\begin{tikzpicture}
			
			\begin{axis}[%
			width=3.05in,
			height=2.8in,
			scale only axis,
			at={(0.772in,0.481in)},
			scale only axis,
			xmin=0.8,
			xlabel={\large localization parameter $\ell$},
			xmax=6.2,
			ymode=log,
			ymin=1e-06,
			ymax=3,
			ylabel={\large rel error in $\|\cdot\|_a$},
			yminorticks=true,
			axis background/.style={fill=white},
			legend style={legend cell align=left, align=left, draw=white!15!black},
			legend pos = south west
			]
			
			\addplot [color=myBlue, mark=square, line width=1.5pt, mark size=4.0pt]
			table[row sep=crcr]{%
				1.000000000000000   0.960426123713420\\
				2.000000000000000   0.601439834558226\\
				3.000000000000000   0.139685347605383\\
				4.000000000000000   0.023507930653259\\
				5.000000000000000   0.003880575672885\\
				6.000000000000000   0.000655064993316\\
			};
			\addlegendentry{$p=1$, $H=2^{-5}$}
			
			\addplot [color=myRed, mark=o, line width=1.5pt, mark size=4.0pt]
			table[row sep=crcr]{%
				1.000000000000000   0.962946841613927\\
				2.000000000000000   0.524232755584708\\
				3.000000000000000   0.085470782339313\\
				4.000000000000000   0.010875812596435\\
				5.000000000000000   0.001370249726825\\
				6.000000000000000   0.000172089153917\\
			};
			\addlegendentry{$p=2$, $H=2^{-5}$}
			
			\addplot [color=myGreen, mark=triangle, line width=1.5pt, mark size=4.0pt]
			table[row sep=crcr]{%
				1.000000000000000   0.968884556209862\\
				2.000000000000000   0.511211609461317\\
				3.000000000000000   0.070694806958015\\
				4.000000000000000   0.007718785636847\\
				5.000000000000000   0.000835248134753\\
				6.000000000000000   0.000090135487256\\
			};
			\addlegendentry{$p=3$, $H=2^{-5}$}
			
			\addplot [color=black, dashed,line width=1.5pt]
			table[row sep=crcr]{%
				1	0.0676676416183064\\
				2	0.00915781944436709\\
				3	0.00123937608833318\\
				4	0.000167731313951256\\
				5	2.26999648812424e-05\\
				6	3.0721061766641e-06\\
			};
			\addlegendentry{$\exp(-2\ell)$}
			
			\end{axis}
			\end{tikzpicture}%
		}
	\end{center}
	\caption[Errors of the high-order multiscale method w.r.t~$\ell$]
	{\small Errors of the high-order multiscale method in the relative energy norm for the first (left) and second model (right) with respect to $\ell$ for different values of $H$ and $p$.}
	\label{fig:errors2dEll}
\end{figure}

For comparison, Figure~\ref{fig:errors2d} also included the error curves for the classical continuous $Q_1$ finite element method as well as the first-order conforming LOD approach as in~\cite{HenP13}. The finite element method does not provide reasonable approximations in the regime where the fine oscillations of the coefficients are not resolved and leads to a stagnation relatively quickly, whereas the LOD approach shows a convergence rate that is even slightly better than predicted by the theory. Still, our multiscale approach with $p=1$ shows a higher convergence rate provided that $\ell$ is chosen appropriately. We emphasize that for a direct comparison of the two methods, one has to keep in mind that for the same mesh size $H$, our multiscale approach with $p=1$ has roughly $2^d$ times more degrees of freedom than the corresponding LOD method. 

For completeness, we present the errors of our multiscale method also with respect to the localization parameter $\ell$ in Figure~\ref{fig:errors2dEll}. The plots show the exponential convergence rate in $\ell$ as in the theory. The curves stagnate for larger values of $\ell$ where the localization error is small enough and the first term in the estimate \eqref{eq:locerrorP} dominates the overall error. 

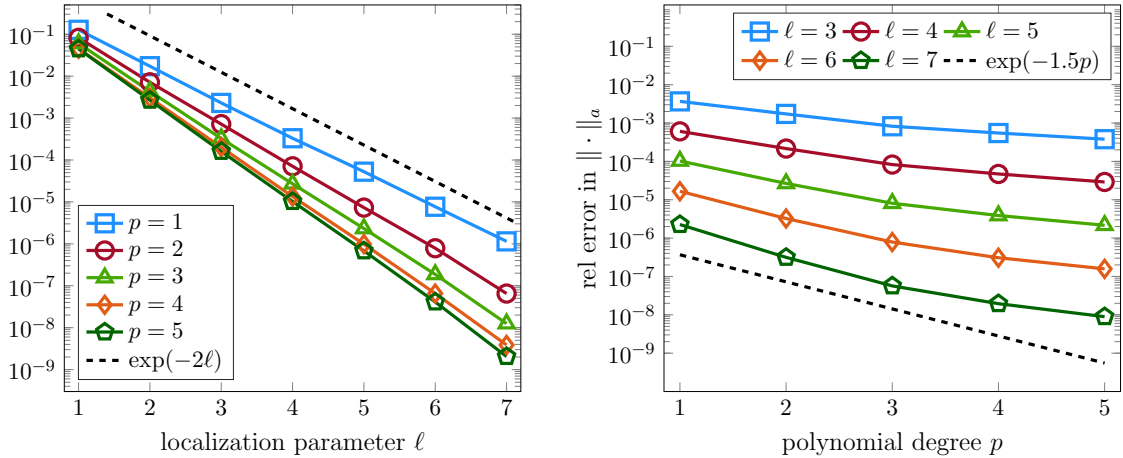
\begin{figure}
	\begin{center}
		\scalebox{0.81}{
			\begin{tikzpicture}
			
			\begin{axis}[%
			width=2.95in,
			height=2.5in,
			at={(0.772in,0.481in)},
			scale only axis,
			xmin=0.8,
			xmax=7.2,
			xlabel={\large localization parameter $\ell$},
			ymode=log,
			ymin=3e-10,
			ymax=0.5,
			yminorticks=true,
			axis background/.style={fill=white},
			legend style={legend cell align=left, align=left, draw=white!15!black},
			legend pos= south west
			]
			\addplot [color=myBlue, mark=square, line width=1.5pt, mark size=4.0pt]
			table[row sep=crcr]{%
				1.000000000000000e+00     1.276408979120657e-01\\
				2.000000000000000e+00     1.728669355055597e-02\\
				3.000000000000000e+00     2.299305829567870e-03\\
				4.000000000000000e+00     3.305971886024076e-04\\
				5.000000000000000e+00     5.269146931929767e-05\\
				6.000000000000000e+00     7.680955038354057e-06\\
				7.000000000000000e+00     1.146534907917991e-06\\
			};
			\addlegendentry{$p=1$}
			
			\addplot [color=myRed, mark=o, fill = white, line width=1.5pt, mark size=4.0pt]
			table[row sep=crcr]{%
				1.000000000000000e+00     8.222068997408045e-02\\
				2.000000000000000e+00     7.146578193815145e-03\\
				3.000000000000000e+00     7.186213703627812e-04\\
				4.000000000000000e+00     7.093845605192213e-05\\
				5.000000000000000e+00     7.325399486586540e-06\\
				6.000000000000000e+00     7.865803488377733e-07\\
				7.000000000000000e+00     6.526349775248619e-08\\
			};
			\addlegendentry{$p=2$}
			
			\addplot [color=myGreen, mark=triangle, line width=1.5pt, mark size=4.0pt]
			table[row sep=crcr]{%
				1.000000000000000e+00     6.025907874518150e-02\\
				2.000000000000000e+00     4.346783590381866e-03\\
				3.000000000000000e+00     3.269061754676900e-04\\
				4.000000000000000e+00     2.710203942646042e-05\\
				5.000000000000000e+00     2.320569771286898e-06\\
				6.000000000000000e+00     1.865936529052594e-07\\
				7.000000000000000e+00     1.251168447270370e-08\\
			};
			\addlegendentry{$p=3$}
			
			\addplot [color=myOrange, mark=diamond, line width=1.5pt, mark size=4.0pt]
			table[row sep=crcr]{%
				1.000000000000000e+00     4.524920161446414e-02\\
				2.000000000000000e+00     3.034707318456008e-03\\
				3.000000000000000e+00     2.017084802613536e-04\\
				4.000000000000000e+00     1.343769274181286e-05\\
				5.000000000000000e+00     9.841746422128850e-07\\
				6.000000000000000e+00     6.537390902196184e-08\\
				7.000000000000000e+00     3.870393582386785e-09\\
			};
			\addlegendentry{$p=4$}
			
			\addplot [color=color3, mark=pentagon, line width=1.5pt, mark size=4.0pt]
			table[row sep=crcr]{%
				1.000000000000000e+00     4.473439025059622e-02\\
				2.000000000000000e+00     2.665787388862203e-03\\
				3.000000000000000e+00     1.602830695335256e-04\\
				4.000000000000000e+00     1.023185832095829e-05\\
				5.000000000000000e+00     6.795324615883041e-07\\
				6.000000000000000e+00     4.106951706333585e-08\\
				7.000000000000000e+00     2.054051126203781e-09\\
			};
			\addlegendentry{$p=5$}
			
			\addplot [color=black, dashed,line width=1.5pt]
			table[row sep=crcr]{%
				1.4	0.30405031312608988\\
				7.1	3.4039906719881712e-06\\
			};
			\addlegendentry{$\exp(-2\ell)$}
			
			\end{axis}
			\end{tikzpicture}%
		}
		\hfill
		\scalebox{0.81}{
			\begin{tikzpicture}
			
			\begin{axis}[%
			width=2.95in,
			height=2.5in,
			at={(0.772in,0.481in)},
			scale only axis,
			xtick={1,2,3,4,5},
			ytick={1e-9,1e-8,1e-7,1e-6,1e-5,1e-4,1e-3,1e-2,1e-1},
			xmin=0.85,
			xmax=5.15,
			xlabel={\large polynomial degree $p$},
			ymode=log,
			ymin=1e-10,
			ymax=1.2,
			yminorticks=true,
			ylabel={\large rel error in $\|\cdot\|_a$},
			axis background/.style={fill=white},
			legend style={legend cell align=left, align=left,legend columns=3, draw=white!15!black}
			]

			\addplot [color=myBlue, mark=square, line width=1.5pt, mark size=4.0pt]
			table[row sep=crcr]{%
				1.000000000000000e+00     3.672989482469776e-03\\
				2.000000000000000e+00     1.730266066992126e-03\\
				3.000000000000000e+00     8.253571602896247e-04\\
				4.000000000000000e+00     5.493533811245637e-04\\
				5.000000000000000e+00     3.796424989145964e-04\\
			};
			\addlegendentry{$\ell=3$}

			\addplot [color=myRed, mark=o, fill = white, line width=1.5pt, mark size=4.0pt]
			table[row sep=crcr]{%
				1.000000000000000e+00     6.070825424920125e-04\\
				2.000000000000000e+00     2.162428994343363e-04\\
				3.000000000000000e+00     8.272933732343173e-05\\
				4.000000000000000e+00     4.707318176770397e-05\\
				5.000000000000000e+00     2.909049991962444e-05\\
			};
			\addlegendentry{$\ell=4$}

			\addplot [color=myGreen, mark=triangle, line width=1.5pt, mark size=4.0pt]
			table[row sep=crcr]{%
				1.000000000000000e+00     1.025201493010978e-04\\
				2.000000000000000e+00     2.666405683965315e-05\\
				3.000000000000000e+00     8.106062239348244e-06\\
				4.000000000000000e+00     3.882773146683274e-06\\
				5.000000000000000e+00     2.163903221022366e-06\\
			};
			\addlegendentry{$\ell=5$}

			\addplot [color=myOrange, mark=diamond, line width=1.5pt, mark size=4.0pt]
			table[row sep=crcr]{%
				1.000000000000000e+00     1.662932300738938e-05\\
				2.000000000000000e+00     3.245819179639724e-06\\
				3.000000000000000e+00     7.898182308741890e-07\\
				4.000000000000000e+00     3.081489860952445e-07\\
				5.000000000000000e+00     1.579674156802406e-07\\
			};
			\addlegendentry{$\ell=6$}
			
			\addplot [color=color3, mark=pentagon, line width=1.5pt, mark size=4.0pt]
			table[row sep=crcr]{%
				1.000000000000000e+00     2.260590238061900e-06\\
				2.000000000000000e+00     3.149795407169433e-07\\
				3.000000000000000e+00     5.676159255781229e-08\\
				4.000000000000000e+00     1.954554077657083e-08\\
				5.000000000000000e+00     8.945464735157870e-09\\
			};
			\addlegendentry{$\ell=7$}
			
			\addplot [color=black, dashed,line width=1.5pt]
			table[row sep=crcr]{%
				1	3.6787944117144229e-07\\
				5	5.5308437014783356e-10\\
			};
			\addlegendentry{$\exp(-1.5p)$}
			
			\end{axis}
			\end{tikzpicture}%
		}
	\end{center}
	\caption[Localization errors of high-order multiscale basis functions]
	{\small Localization errors of the high-order multiscale basis functions on the scale $H = 2^{-4}$ for the first model with respect to $\ell$ (left) and $p$ (right) in the relative energy norm.} 
	\label{fig:decaybasisfunc}
\end{figure}

Since the previous experiments indicate that the exponential convergence in $\ell$ even slightly improves when $p$ is increased, we further investigate the sharpness of the decay estimate quantified in Theorem~\ref{t:decay}. To this end, for $H = 2^{-4}$ we choose the element $K = [0.4375,0.5]^2\in \tri_H$ and compute the relative energy error between the ideal multiscale basis functions $\tLamKj := \calR_h \LamKj$ and its localized versions $\tLamKjl := \calR^\ell_h \LamKj$ for different values of $\ell$ and $j\in\{1,\ldots,\nK\}$.
For the first model, Figure~\ref{fig:decaybasisfunc}~(left) shows the decay of the localization error for different basis functions with respect to $\ell$. To be more precise, for each~$p$ we show the localization error corresponding to the highest-order basis function $\LamKj$ (with maximal polynomial degree $p$ in both components). The results seem to contradict the scaling in $p$ as predicted by Theorem~\ref{t:decay}. Instead, the rate even slightly improves when the polynomial degree $p$ is increased.  
In Figure~\ref{fig:decaybasisfunc}~(right), we show the localization error for different $\ell$ and $p$ corresponding to the respective lowest-order basis function, i.e., the one whose $L^2$-projection onto $\VHp(K)$ is constant. Again, the curves show an error reduction when $p$ is increased which is slightly amplified by $\ell$. That is, these results also indicate a better scaling in $p$ than quantified in Theorem~\ref{t:decay}. The commencing stagnation of the errors in Figure~\ref{fig:decaybasisfunc}~(right) for larger $p$ is probably related to the fact that $h = 2^{-9}$ is not fine enough to handle higher polynomial degrees. 
For further numerical experiments, see also \cite[Sec.~3.4]{Mai20}

\subsection*{Discussion}
The numerical experiments of this section overall confirm the theoretical results for our high-order multiscale method. The only deviation is in the scaling with respect to the polynomial degree $p$ which seems to be better than predicted by the theory. That is, the result presented in Theorem~\ref{t:decay} is most likely not sharp with respect to $p$ due to the mismatch between interpolation estimates and inverse inequalities as mentioned in Remark~\ref{r:pdep}. An enhanced estimate would directly relax the condition on $\ell$ which is quantified in \eqref{eq:lscale}.

Note that although the approach numerically and theoretically shows a pollution of the total error for small mesh sizes $H$, this issue can be compensated for by a correct scaling of $\ell$. Nevertheless, the method shows its best potential for relatively coarse mesh sizes which, combined with higher-order polynomials, already provide very good approximation results. Moreover, the locality of the high-order construction in principle allows us to even choose different polynomial degrees on each of the coarse elements dependent, e.g., on the local regularity of $f$. 

\section{Conclusion}\label{s:conclusion} 
Within this paper, we have considered an elliptic model problem with possibly varying (fine-scale) diffusion coefficient. We have proposed a multiscale technique motivated by the LOD method and gamblets that is able to achieve high-order convergence rates with respect to the mesh size and the polynomial degree independently of oscillations of the coefficient. The method can be applied for problems involving general unstructured coefficients and only requires minimal regularity assumptions on the domain, the diffusion coefficient, and the exact solution. We have proved decay estimates for the multiscale basis functions and suggested an appropriate localization strategy to reduce the computational complexity. Numerical experiments confirm the theoretical findings and even indicate a better behavior with respect to the polynomial degree than theoretically predicted.

%
%
\section*{Acknowledgments} 
R.~Maier acknowledges support by the German Research Foundation (DFG) in the Priority Program 1748 \emph{Reliable simulation techniques in solid mechanics} (PE2143/2-2). Further, the author thanks D.~Peterseim for providing large support in the code development. 
%
%

%
%
\end{document}